\documentclass[10pt]{amsart}
\makeatletter
\tagsleft@false
\makeatother

\usepackage{amssymb}

\usepackage{enumitem}

\newtheorem{thm}{Theorem}[section]
\newtheorem{prop}[thm]{Proposition}
\newtheorem{lem}[thm]{Lemma}
\newtheorem{cor}[thm]{Corollary}
\newtheorem{quest}[thm]{Question}

\makeatletter
\newtheorem*{rep@theorem}{\rep@title}
\newcommand{\newreptheorem}[2]{%
\newenvironment{rep#1}[1]{%
 \def\rep@title{#2 \ref{##1}}%
 \begin{rep@theorem}}%
 {\end{rep@theorem}}}
\makeatother

\newreptheorem{theorem}{Theorem}

\theoremstyle{definition}

\numberwithin{equation}{section}

\newcommand{\lut}{\operatorname{LUT}^1}
\newcommand{\fleh}{\operatorname{FLE}}

\newcommand{\h}{\operatorname{H}}
\newcommand{\splay}{\operatorname{Splay}}
\newcommand{\fullsplay}{\operatorname{FullSplay}}
\newcommand{\cone}{\operatorname{Cone}}
\newcommand{\e}{\mathcal{E}}

\usepackage{tkz-graph}\usetikzlibrary{decorations.markings}\usetikzlibrary{quotes}
\tikzset{->-/.style={
 decoration={
   markings,
   mark=at position #1 with {\arrow{>}}
   },
   postaction={decorate}
  },
  ->-/.default=0.5 
}
\usetikzlibrary{positioning}

\def\PlusCross{{%
    \setbox0\hbox{$+$}%
    \rlap{\hbox to \wd0{\hss$\times$\hss}}\box0
}}

\newcommand{\setup}{\GraphInit[vstyle=Empty]
            \SetVertexSimple[MinSize = 1pt]
            \SetUpEdge[lw = 0.5pt]
            \tikzset{EdgeStyle/.style={->-}}
            \tikzset{VertexStyle/.append style = {minimum size = 3pt, inner sep = 0pt}}
            \SetVertexNoLabel
            \SetGraphUnit{2}}

\DeclareFontFamily{U}{mathx}{\hyphenchar\font45}
\DeclareFontShape{U}{mathx}{m}{n}{
      <5> <6> <7> <8> <9> <10>
      <10.95> <12> <14.4> <17.28> <20.74> <24.88>
      mathx10
      }{}
\DeclareSymbolFont{mathx}{U}{mathx}{m}{n}
\DeclareMathSymbol{\bigtimes}{1}{mathx}{"91}

\calclayout

\usepackage{hyperref}

\usepackage{faktor}

\usepackage{xcolor}
\usepackage{soul}

\begin{document}

\title[Free left $h$-Ehresmann semigroups]{Free left $h$-Ehresmann semigroups}

\author{Daniel Heath}

\begin{abstract}

The class of left $h$-adequate semigroups sits between the classes of left ample and left Ehresmann semigroups. Whilst free objects in related classes have descriptions in terms of directed trees, free left $h$-adequate semigroups have yet only been described in terms of Rees quotients of free products. Here, we reintroduce the class of (left) $h$-Ehresmann semigroups, show their free objects coincide with free (left) $h$-adequate semigroups, and give a description of these free objects in terms of directed trees. We use our new description in investigating various properties of these structures, including finitary conditions of recent interest.

\end{abstract}

\maketitle



\section{Introduction}
Some of the earliest structures investigated in the \textit{York school of semigroup theory} were Fountain's (\textit{left}) \textit{adequate} semigroups \cite{fountain:adequate}. Informally, these are semigroups whose (right) cancellation properties are exactly determined by their semilattice of idempotents. They generalise the well-known (and well-studied) class of \textit{inverse} semigroups which were gaining interest at the time partly due to Munn's graphical description of free inverse monoids \cite{munn:freeinv}. They have since been generalised themselves with \textit{Ehresmann} semigroups, which were coined by Lawson \cite{lawson:saoc} due to their connection with work of Ehresmann on small categories \cite{ehresmann:small}.

Ehresmann and adequate semigroups have received interest over the subsequent decades (see e.g. \cite{aird:growth,batbedat:connections,branco:ehresmann,heath:graph,kambites:free,kambites:freeleft}), though much of the research involves the closely related class of \textit{restriction} semigroups and their subclass of (\textit{left}) \textit{ample} semigroups (see e.g. \cite{fountain:freeample,gould:expansions,gould:restr,gould:coherentfreemons}), originally coined (\textit{left}) \textit{type A} semigroups in \cite{fountain:adequate}.

It is perhaps lesser known that the original article \cite{fountain:adequate} also defines (\textit{right}) \textit{type B} semigroups. These are right adequate semigroups satisfying an additional identity which we call the \textit{right $h$ identity}, along with an extra condition on idempotents. Despite right type B semigroups and right ample semigroups forming disjoint classes \cite[Example 3.3]{fountain:adequate}, right ample semigroups do satisfy the $h$ identity. In a subsequent paper \cite{fountain:h}, Fountain weakens the right ample condition and considers right adequate semigroups satisfying the right $h$ identity: such structures are coined \textit{right $h$-adequate} semigroups. In \cite{fountain:h}, their free structures are described in terms of free products, and classical semigroup-theoretic properties of them are examined, such as $\mathcal{J}$-triviality (in the sense of \cite{green:relations}), solvability of the word problem, and having free semigroups as maximal left cancellative image. Right Ehresmann/adequate semigroups satisfying the right $h$ identity were further investigated under the name (\textit{weakly}) \textit{right $E$-hedged} in \cite{fountain:munntype,gomes:fund}.

Similar to Munn's description for inverse semigroups, free (left) Ehresmann and free (left) adequate semigroups (which coincide) and free (left) ample and free (left) restriction semigroups (which also coincide) have been given graphical descriptions \cite{fountain:freeample,gould:expansions,kambites:free,kambites:freeleft}. Recently, Aird and the author \cite{aird:growth} investigated the free left Ehresmann/adequate monoid of rank $1$ using this description and many of the results therein rely on it satisfying the dual of the right $h$ identity. It follows that the free left Ehresmann monoid of rank $1$ and the free left $h$-adequate monoid of rank $1$ coincide. This naturally leads to the question of whether the description of free left and right $h$-adequate monoids given by Fountain \cite{fountain:h} can alternatively be given in terms of directed trees in all ranks.

Here, we reintroduce \textit{left}/\textit{right}/\textit{two-sided} \textit{$h$-Ehresmann} monoids (formerly weakly hedged monoids \cite{fountain:munntype,gomes:fund}) and investigate their free structures. We show that free left $h$-Ehresmann monoids and free left $h$-adequate monoids coincide in all ranks, and give a description of them in terms of directed trees similar to results for free Ehresmann, ample and inverse monoids. This involves a new operation on trees which we term \textit{splaying} (see Section \ref{sec:splaying}).

The present paper proceeds as follows. In Section \ref{sec:prelims} we give some necessary background on (left) Ehresmann and (left) adequate semigroups, including Kambites' description of free objects in terms of trees \cite{kambites:freeleft}. In Section \ref{sec:hEhr}, we discuss left, right, and two-sided $h$-Ehresmann semigroups and provide some motivating examples. We extend a recent result of Gould and Johnson \cite[Corollary 6.6]{gould:coherent} and show that free left Ehresmann monoids are right $h$-adequate (Theorem \ref{thm:flehisrhad}). In Section \ref{sec:splaying}, we give a description of free left $h$-adequate monoids in terms of directed trees, and use it in Section \ref{sec:props} to investigate various properties they enjoy. In particular we show that the free left $h$-Ehresmann monoids coincide with the free left $h$-adequate monoids (Corollary \ref{cor:hEhrandhAdcoincide}), are two-sided $h$-adequate (Theorem \ref{thm:hishadequate}) and satisfy various finitary properties of recent interest such as being \textit{weakly coherent}, \textit{finitely equated} and \textit{ideal Howson} (Section \ref{sec:finitary}). We close by discussing some open questions.

\section{Preliminaries}\label{sec:prelims}

We assume the reader is familiar with general semigroup theory and fundamentals of universal algebra. Whilst we cover necessary details here, we refer the reader to \cite{howie:fundamentals} for details on semigroups and \cite{burris:universal} for details on universal algebra. 

\subsection{Ehresmann semigroups}

A semigroup $S$ equipped with a unary operation ${}^+: S \to S$ is called a \textit{left Ehresmann semigroup} if
\[x^+x = x,\quad(x^+y^+)^+ = x^+y^+ = y^+x^+\mkern18mu\text{and}\mkern18mu(xy)^+=(xy^+)^+\]for all $x,y \in S$.

A semigroup $S$ equipped with a unary operation ${}^\ast: S \to S$ is called a \textit{right Ehresmann semigroup} if
\[xx^\ast = x,\quad(x^\ast y^\ast)^\ast = x^\ast y^\ast = y^\ast x^\ast \mkern18mu\text{and}\mkern18mu (xy)^\ast=(x^\ast y)^\ast\]for all $x,y \in S$.

A semigroup $S$ equipped with two unary operations ${}^+$ and ${}^\ast$ is called a (\textit{two-sided}) \textit{Ehresmann semigroup} if it is left Ehresmann with respect to ${}^+$, right Ehresmann with respect to ${}^\ast$, and such that it satisfies the \textit{linking identities}:
\[(x^+)^\ast = x^+\mkern18mu\text{and}\mkern18mu (x^\ast)^+ = x^\ast.\]

We will always treat these structures as coming equipped with their unary operation(s). In line with other literature \cite{aird:growth,branco:ehresmann,kambites:free,kambites:freeleft}, we consider left Ehresmann semigroups and right Ehresmann semigroups as $(2,1)$-algebras and Ehresmann semigroups as $(2,1,1)$-algebras with their supplementary operations.

If a [resp. left/right/two-sided] Ehresmann semigroup contains a multiplicative identity, we say it is a [resp. \textit{left}/\textit{right}/\textit{two-sided}] \textit{Ehresmann monoid} and enrich the signature to $(2,1,0)$ in the one-sided cases or $(2,1,1,0)$ in the two-sided case by distinguishing the identity element. It follows from the axioms above that the identity element is stabilised by the unary operation(s).

These classes contain a wealth of examples: indeed, any monoid $M$ is Ehresmann under the unary operations $m^+= m^\ast = 1$ for all $m \in M$. We call such examples \textit{reduced Ehresmann monoids}. \textit{Inverse} semigroups are also Ehresmann: for an inverse semigroup $S$, one can take $s^+ = ss^{-1}$ and $s^\ast = s^{-1}s$. We will observe more examples in Section \ref{sec:examples}.

In studying a given left Ehresmann semigroup $S$, the set of \textit{projections}\[P(S) = \left\{s^+ \mid s \in S \right\}\]consistently proves its importance. Projections of a right Ehresmann semigroup are defined dually. In the case of a two-sided Ehresmann semigroup, the linking identities ensure that the sets coincide:
\[P(S) = \left\{s^+ \mid s \in S \middle\}\, =\, \middle\{s^\ast \mid s \in S \right\}.\]In all cases, $P(S)$ forms a subsemilattice of $S$ where $e \leq f$ if and only if $e = ef$. We highlight the important observation that, whilst all projections are idempotent, idempotents are not necessarily projections.

\subsection{Adequate semigroups}
We mention an important subclass of Ehresmann semigroups. A left Ehresmann semigroup is called \textit{left adequate} if it additionally satisfies the following quasi-identities:\[x^2 = x \rightarrow x = x^+ \mkern18mu\text{and}\mkern18mu xz = yz \rightarrow xz^+ = yz^+.\]Dually, a right Ehresmann semigroup is called \textit{right adequate} if it additionally satisfies the following quasi-identities:\[x^2 = x \rightarrow x = x^\ast \mkern18mu\text{and}\mkern18mu zx = zy \rightarrow z^\ast x = z^\ast y.\]An Ehresmann semigroup is called (\textit{two-sided}) \textit{adequate} if it is left adequate and right adequate. We note here that the linking identities automatically follow from the quasi-identities above, and thus it is enough to separately show left and right adequacy of a given semigroup to deduce two-sided adequacy. Moreover in both one-sided and two-sided adequate semigroups, the quasi-identities further ensure that all idempotents are projections.

In the adequate cases, the unary operations ${}^+$ and ${}^\ast$ are uniquely determined: if $S$ is, say, a left adequate semigroup with respect to a unary operation ${}^{+_1}$ and an operation ${}^{+_2}$, then $s^{+_1} = s^{+_2}$ for all $s \in S$. In this sense, we may say a semigroup is \textit{left adequate} (with no mention of a unary operation) if there is some unary operation under which it is left adequate. We use similar terms for right and two-sided adequate.

One may equivalently define adequate semigroups using Green-esque relations $\mathcal{L}^\ast$ and $\mathcal{R}^\ast$ \cite{fountain:adequate}, and study the wider class of Ehresmann semigroups through analogues $\widetilde{\mathcal{L}}$ and $\widetilde{\mathcal{R}}$ \cite{elqallali:struc,gould:restr}, though we do not take these routes here.

\subsection{Free structures}

The present paper is primarily concerned with free structures in various classes. We recall here some fundamentals of universal algebra which we require going forward; for a more detailed overview and undefined terms see \cite{burris:universal} or \cite[Section 2.2]{heath:graph}.  

Let $\mathcal{C}$ be a class of algebras of common signature. An algebra $F \in \mathcal{C}$ is called \textit{free} on a subset $X \subseteq F$ if for every $M \in \mathcal{C}$ and every function $\chi :X \to M$, there is a unique \textit{morphism} $F \to M$ (i.e. a map respecting all operations) extending $\chi$.

By classical results of universal algebra, if the class $\mathcal{C}$ forms a non-trivial \textit{quasi-variety} then a free object $F$ exists for any set $X$ and is unique up to isomorphism \cite{birkhoff:struct,burris:universal}. Moreover if $X$ and $Y$ are sets of the same cardinality, the unique free object on $X$ and the unique free object on $Y$ are isomorphic \cite[Lemma 4.110]{mckenzie:algebra}. We say this unique free object is of \textit{rank} $|X|$ (equivalently $|Y|$).

The classes of one-sided and two-sided Ehresmann/adequate semigroups all form quasi-varieties in their respective signatures (in fact \textit{varieties} in the Ehresmann cases). Hence free objects exist in all ranks. Moreover, it is seen that for any set $X$, the free [resp. left/right/two-sided] Ehresmann semigroup on $X$ and the free [resp. left/right/two-sided] adequate semigroup on $X$ coincide (see \cite[Section 6]{kambites:free}). These free objects were given a description by Kambites \cite{kambites:free,kambites:freeleft} via birooted trees. 

\subsection{\texorpdfstring{$X$}{X}-trees}\label{sec:trees}

Here we recall the description of free left Ehresmann monoids (of rank at least $1$) from \cite{kambites:freeleft}. Fix a non-empty set $X$. A \textit{left $X$-tree} $\Gamma$ is a finite directed tree, edge-labelled by elements of $X$, with two distinguished vertices (called \textit{start} and \textit{end}) such that for every vertex $v$ of $\Gamma$ there is a (possibly empty, necessarily unique) directed path from the start vertex to $v$. The directed path from the start vertex to the end vertex is called the \textit{trunk}. We call the vertices and edges of the trunk the \textit{trunk vertices} and \textit{trunk edges} respectively.

We say an edge from a vertex $v_1$ to a vertex $v_2$ has \textit{initial} vertex $v_1$, \textit{terminal} vertex $v_2$, and is \textit{incident} to both $v_1$ and $v_2$. We say a vertex $v$ is \textit{reachable} from a vertex $u$ if there is some directed path from $u$ to $v$. Note that every vertex is reachable from the start vertex in a left $X$-tree.

We extend the labelling on edges to paths: the label of a path is the word over $X$ given by the concatenation of its edge labels. We say a word $w$ is \textit{readable} from a vertex $u$ to a vertex $v$ if there is a directed path with initial vertex $u$, terminal vertex $v$, and label $w$.

We represent left $X$-trees diagrammatically as graphs in the expected way, with the symbol {\large $+$} decorating the start vertex and the symbol {\large $\times$} decorating the end vertex. In instances where the start and end vertex coincide, we decorate using the symbol {\large$\PlusCross$}. Figure \ref{fig:xtrees} shows some examples of left $\{x,y\}$-trees.

\begin{figure}[ht]
    \centering
        \begin{tikzpicture}
            \setup

            \node (A) at (0,0) {\large$+$};
            \Vertex[x=0,y=1]{C1};
            \Vertex[x=1,y=1]{B1};
            \Vertex[x=0,y=2]{C2};
            \Vertex[x=1,y=2]{B2a};
            \Vertex[x=2,y=3]{B2aa};
            \Vertex[x=2,y=1]{B2ab};
            \node (C3) at (0,3) {\large$\times$};
        
            \Edge(A)(C1)\draw (A) -- (C1) node[midway,left=2pt] {$x$};
            \Edge(C1)(B1)\draw (C1) -- (B1) node[midway,below=2pt] {$x$};
            \Edge(C1)(C2)\draw (C1) -- (C2) node[midway,left=2pt] {$y$};
            \Edge(C2)(B2a)\draw (C2) -- (B2a) node[midway,below=2pt] {$x$};
            \Edge(B2a)(B2aa)\draw (B2a) -- (B2aa) node[midway,left=2pt] {$y$};
            \Edge(B2a)(B2ab)\draw (B2a) -- (B2ab) node[midway,left=2pt] {$x$};
            \Edge(C2)(C3)\draw (C2) -- (C3) node[midway, left=2pt] {$x$};
        \end{tikzpicture}
        \hspace{3em}
        \begin{tikzpicture}
            \setup

            \node (A) at (0,0) {\large$+$};
            \Vertex[x=0,y=1]{C1};
            \Vertex[x=1,y=2]{B1};
            \Vertex[x=0,y=3]{C3};
            \Vertex[x=1,y=3]{B2aa};
            \Vertex[x=2,y=2]{B2ab};
            \node (C2) at (0,2) {\large$\times$};
        
            \Edge(A)(C1)\draw (A) -- (C1) node[midway,left=2pt] {$x$};
            \Edge(C1)(B1)\draw (C1) -- (B1) node[midway,below=2pt] {$y$};
            \Edge(C1)(C2)\draw (C1) -- (C2) node[midway,left=2pt] {$y$};
            \Edge(B1)(B2aa)\draw (B1) -- (B2aa) node[midway,left=2pt] {$x$};
            \Edge(B1)(B2ab)\draw (B1) -- (B2ab) node[midway,below=2pt] {$y$};
            \Edge(C2)(C3)\draw (C2) -- (C3) node[midway, left=2pt] {$y$};
        \end{tikzpicture}
        \hspace{3em}
        \begin{tikzpicture}
            \setup

            \node (_) at (0,0) {};
            \node (A) at (0,0.5) {\large$\PlusCross$};
            \Vertex[x=0,y=1.5]{C1};
            \Vertex[x=1,y=0.5]{B0};
            \Vertex[x=1,y=2.5]{B1};
            \Vertex[x=0,y=2.5]{C2};
        
            \Edge(A)(C1)\draw (A) -- (C1) node[midway,left=2pt] {$y$};
            \Edge(A)(B0)\draw (A) -- (B0) node[midway,above=2pt] {$x$};
            \Edge(C1)(B1)\draw (C1) -- (B1) node[midway,below=2pt] {$x$};
            \Edge(C1)(C2)\draw (C1) -- (C2) node[midway,left=2pt] {$x$};

        \end{tikzpicture}
    \caption{Some left $X$-trees where $X = \{x,y\}$.}
    \label{fig:xtrees}
\end{figure}
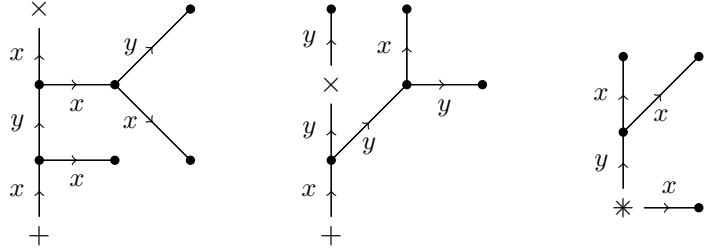

A \textit{morphism} of left $X$-trees from $\Gamma$ to $\Delta$ is a directed graph morphism $\phi:\Gamma \to \Delta$ respecting edge-labelling and mapping the start vertex of $\Gamma$ to the start vertex of $\Delta$ and likewise for the end vertex. We will, without exception, consider left $X$-trees only up to \textit{isomorphism} and consider trees equal if they are formally isomorphic.

We denote the set of all left $X$-trees (up to isomorphism) by $\lut(X)$. Given two left $X$-trees $\Gamma,\Delta \in \lut(X)$, Kambites \cite{kambites:free} defines the \textit{gluing product} $\Gamma \times \Delta$ to be the left $X$-tree obtained by gluing $\Delta$ to $\Gamma$ start-to-end, with start vertex the start vertex of $\Gamma$ and end vertex the end vertex of $\Delta$. Further, the left $X$-tree $\Gamma^{(+)}$ is defined to be the tree $\Gamma$ with the end vertex instead located at the start vertex. Figure \ref{fig:operations} shows an example of these two operations.

\begin{figure}[ht]
    \centering
        \begin{tikzpicture}
            \setup

            \node (_) at (0,0) {};
            \node (A) at (0,0.5) {\large$+$};
            \Vertex[x=0,y=1.5]{C1};
            \Vertex[x=1,y=1.5]{B1};
            \node (C2) at (0,2.5) {\large$\times$};
            \Vertex[x=1,y=2.5]{B2};
        
            \Edge(A)(C1)\draw (A) -- (C1) node[midway,left=2pt] {$x$};
            \Edge(C1)(C2)\draw (C1) -- (C2) node[midway,left=2pt] {$y$};
            \Edge(C1)(B1)\draw (C1) -- (B1) node[midway,below=2pt] {$x$};
            \Edge(C2)(B2)\draw (C2) -- (B2) node[midway,below=2pt] {$x$};

        \end{tikzpicture}
        \hspace{3em}
        \begin{tikzpicture}
            \setup

            \node (_) at (0,0) {};
            \node (A) at (0,1) {\large$+$};
            \Vertex[x=1,y=1]{B1};
            \node (C1) at (0,2) {\large$\times$};
        
            \Edge(A)(C1)\draw (A) -- (C1) node[midway,left=2pt] {$x$};
            \Edge(A)(B1)\draw (A) -- (B1) node[midway,below=2pt] {$y$};

        \end{tikzpicture}
        \hspace{3em}
        \begin{tikzpicture}
            \setup

            \node (A) at (0,0) {\large$+$};
            \Vertex[x=0,y=1]{C1};
            \Vertex[x=1,y=1]{B1};
            \Vertex[x=0,y=2]{C2};
            \Vertex[x=1,y=1.5]{B2};
        
            \Edge(A)(C1)\draw (A) -- (C1) node[midway,left=2pt] {$x$};
            \Edge(C1)(C2)\draw (C1) -- (C2) node[midway,left=2pt] {$y$};
            \Edge(C1)(B1)\draw (C1) -- (B1) node[midway,below=2pt] {$x$};
            \Edge(C2)(B2)\draw (C2) -- (B2) node[midway,below=2pt] {$x$};

            \Vertex[x=1,y=2.5]{dB1};
            \node (dC1) at (0,3) {\large$\times$};
        
            \Edge(C2)(dC1)\draw (C2) -- (dC1) node[midway,left=2pt] {$x$};
            \Edge(C2)(dB1)\draw (C2) -- (dB1) node[midway,above=2pt] {$y$};

        \end{tikzpicture}
        \hspace{3em}
        \begin{tikzpicture}
            \setup

            \node (_) at (0,0) {};
            \node (A) at (0,0.5) {\large$\PlusCross$};
            \Vertex[x=0,y=1.5]{C1};
            \Vertex[x=1,y=1.5]{B1};
            \Vertex[x=0,y=2.5]{C2};
            \Vertex[x=1,y=2.5]{B2};
        
            \Edge(A)(C1)\draw (A) -- (C1) node[midway,left=2pt] {$x$};
            \Edge(C1)(C2)\draw (C1) -- (C2) node[midway,left=2pt] {$y$};
            \Edge(C1)(B1)\draw (C1) -- (B1) node[midway,below=2pt] {$x$};
            \Edge(C2)(B2)\draw (C2) -- (B2) node[midway,below=2pt] {$x$};

        \end{tikzpicture}
    \caption{From left-to-right, some left $\{x,y\}$-trees $\Gamma$ and $\Delta$, with the product $\Gamma \times \Delta$ and the tree $\Gamma^{(+)}$.}
    \label{fig:operations}
\end{figure}
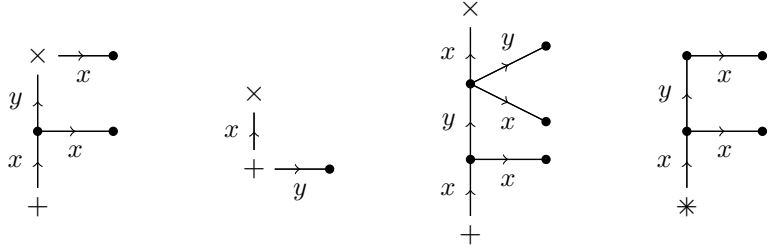

With these operations and distinguished trivial one-vertex tree, $\lut(X)$ forms a $(2,1,0)$-algebra which is generated by the set of single-edge trees with distinct start and end vertices, edge-labelled $x \in X$ for each $x \in X$. Throughout, we will implicitly identify this generating set of $\lut(X)$ with $X$ itself.

\begin{prop}\label{prop:lutgenerated}
    Let $X$ be a non-empty set. Then $\lut(X)$ is $X$-generated as a $(2,1,0)$-algebra.
\end{prop}

\begin{proof}
    The proof is analogous to that of \cite[Proposition 5.1]{kambites:free}.
\end{proof}

The set of all left $X$-trees with coinciding start and end vertex forms a $(2,1,0)$-subalgebra of $\lut(X)$, which we denote $\e(X)$. We stress that neither $\lut(X)$ nor $\e(X)$ is left Ehresmann (the identity $\Gamma^{(+)}\times\Gamma = \Gamma$ fails for any non-trivial $\Gamma$ for example).

\subsection{Retractions}

A \textit{retraction} on a left $X$-tree $\Gamma$ is a morphism $\phi: \Gamma \to \Gamma$ such that $\phi \circ \phi = \phi$. A left $X$-tree is called \textit{pruned} (or \textit{retract-free}) if the only retraction on $\Gamma$ is the identity map. Kambites \cite{kambites:free} made the following important observation which will be key for us going forward.

\begin{prop}[{\cite[Proposition 3.5]{kambites:free}}]\label{prop:retractconfluent}
    Let $\Gamma$ be a left $X$-tree. Then there exists a unique pruned left $X$-tree $\overline{\Gamma}$ which is the image of $\Gamma$ under a retraction $\Gamma \to \Gamma$.
\end{prop}

Denote the set of all pruned left $X$-trees by $\operatorname{T}(X)$. One may thus consider $\overline{\,\cdot\,}$ as a map from $\lut(X)$ to $\operatorname{T}(X)$. It follows from Proposition \ref{prop:retractconfluent} that\[\operatorname{T}(X) = \left\{ \overline{\Gamma} \mid \Gamma \in \lut(X) \right\}.\]We imbue $\operatorname{T}(X)$ with a $(2,1,0)$-structure by defining operations\[\Gamma\Delta := \overline{\Gamma \times \Delta}\quad\text{and}\quad\Gamma^+ := \overline{\Gamma^{(+)}}\]and distinguishing the trivial one-vertex tree. These operations interact well with retraction; for example the retraction map $\overline{\,\cdot\,}$ is a surjective morphism of $(2,1,0)$-algebras (see \cite[Theorem 4.5]{kambites:free} and \cite[Proposition 3.2]{kambites:freeleft}). We state here various known results which we will refer to throughout.

\begin{lem}[{\cite[Theorem 4.5, Proposition 4.9]{kambites:free}}]\label{lem:barismorphism}
    Let $\Gamma,\Delta,\Xi \in \lut(X)$. Then:\begin{enumerate}
        \item $\overline{\overline{\Gamma}} = \overline{\Gamma}$.\label{lem:barismorphism:barbar}
        \item $\overline{\Gamma \times \Delta} = \overline{\overline{\Gamma} \times \overline{\Delta}}$.\label{lem:barismorphism:bartimesbar}
        \item $\overline{\Gamma^{(+)}} = \overline{\overline{\Gamma}^{(+)}}$. \label{lem:barismorphism:plusbar}
        \item $\overline{\Gamma \times \Xi} = \overline{\Delta \times \Xi}$ if and only if $\overline{\Gamma \times \Xi^{(+)}} = \overline{\Delta \times \Xi^{(+)}}$.\label{lem:barismorphism:adequatequasi}
    \end{enumerate}
\end{lem}

We are now in a position to state the main result of \cite{kambites:freeleft}.

\begin{thm}[Kambites {\cite[Theorem 3.16, Theorem 3.17]{kambites:freeleft}}]\label{thm:flem}
    Let $X$ be any non-empty set. Then the $(2,1,0)$-algebra $\operatorname{T}(X)$ is the free left Ehresmann/adequate monoid on $X$. Moreover, the subset given by $\operatorname{S}(X) := \operatorname{T}(X) \setminus \{\varepsilon\}$, where $\varepsilon$ is the trivial tree, is a $(2,1)$-algebra and is exactly the free left Ehresmann/adequate semigroup on $X$. These structures are freely generated in their respective signatures by the set of single-edge trees with distinct start and end vertices, edge-labelled by $x$ for each $x \in X$. The set of projections/idempotents is exactly the image of $\e(X)$ under the retraction map, i.e. the pruned left $X$-trees with coinciding start and end vertex.
\end{thm}

There is, of course, a dual to Theorem \ref{thm:flem} for free right Ehresmann/adequate monoids. We will mainly concern ourselves with only the left case throughout. For a set $X$, we henceforth denote the free left Ehresmann monoid on $X$ by $\fleh(X)$.

This description in terms of directed graphs mirrors those for free objects in closely related classes, including free inverse monoids \cite{munn:freeinv} and free ample/birestriction monoids \cite{fountain:freeample}.

\section{\texorpdfstring{$h$}{h}-Ehresmann semigroups}\label{sec:hEhr}

This paper mainly concerns a subclass of (left, right) Ehresmann semigroups. These will be those satisfying certain additional identities considered in \cite{batbedat:connections,batbedat:identities,fountain:h,fountain:munntype,gomes:fund}. Semigroups which satisfy the forthcoming $h$-Ehresmann or $h$-adequate conditions were previously known as \textit{weakly hedged} or \textit{hedged} semigroups respectively \cite{fountain:munntype,gomes:fund}.

\subsection{The \texorpdfstring{$h$}{h} conditions}

Let $(S,\cdot,{}^+)$ be a left Ehresmann semigroup. We say $S$ is \textit{left $h$-Ehresmann} if $S$ \textit{satisfies the left $h$ condition}:\[(xz^+y)^+ = (xy)^+(xz)^+ \text{ for all }x,y,z \in S.\]Let $(S,\cdot,{}^\ast)$ be a right Ehresmann semigroup. We say $S$ is \textit{right $h$-Ehresmann} if $S$ \textit{satisfies the right $h$ condition}:\[(xz^\ast y)^\ast = (xy)^\ast(zy)^\ast \text{ for all }x,y,z \in S.\]We say an Ehresmann semigroup $(S,\cdot,{}^+,{}^\ast)$ is \textit{$h$-Ehresmann} if it satisfies the left $h$ condition as a left Ehresmann semigroup, and the right $h$ condition as a right Ehresmann semigroup.

We say a [resp. left/right/two-sided] adequate semigroup is [resp. \textit{left}/\textit{right}/\textit{two-sided}] \textit{$h$-adequate} if it satisfies the corresponding $h$ condition(s).

One may prefer the following equivalent definitions for $h$-Ehresmann semigroups. The third condition in Proposition \ref{prop:hTFAE} corresponds to the original definition of right $h$-adequate semigroups introduced by Fountain \cite{fountain:h}.

\begin{prop}\label{prop:hTFAE}
    Let $S$ be an Ehresmann semigroup. Then the following are equivalent:
    \begin{enumerate}
        \item $S$ is $h$-Ehresmann.\label{prop:hTFAE:hAd}
        \item $(x z^\ast y)^+ = (xy)^+(xz^\ast)^+$ and $(xz^+y)^\ast = (xy)^\ast(z^+y)^\ast$ for all $x,y,z \in S$.\label{prop:hTFAE:identities}
        \item The functions\[\alpha_x : P(S) \to P(S), \quad e\mapsto (ex)^\ast,\]\[\beta_x: P(S) \to P(S), \quad e\mapsto (xe)^+\]are semigroup morphisms for all $x \in S$.\label{prop:hTFAE:funcs}
    \end{enumerate}
\end{prop}

\begin{proof}
    The equivalence of \eqref{prop:hTFAE:hAd} and \eqref{prop:hTFAE:funcs} is observed in \cite[Lemma 2.7]{fountain:munntype}.
    
    Now suppose \eqref{prop:hTFAE:hAd}: $S$ is $h$-Ehresmann. Let $x,y,z \in S$ be given. Since $(z^\ast)^+ = z^\ast$ and $S$ is in particular left $h$-Ehresmann, we see that $(x z^\ast y)^+ =(x(z^\ast)^+ y)^+ = (xy)^+(xz^\ast)^+$. The second condition of \eqref{prop:hTFAE:identities} is seen dually. Conversely, suppose \eqref{prop:hTFAE:identities}: $(x z^\ast y)^+ = (xy)^+(xz^\ast)^+$ and $(xz^+y)^\ast = (xy)^\ast(z^+y)^\ast$ for all $x,y,z \in S$. For any given $x,y,z \in S$, since $z^+ = (z^+)^\ast$, it is easily seen that \[(xz^+y)^+ = (x(z^+)^\ast y)^+ = (xy)^+(x(z^+)^\ast)^+ = (xy)^+(xz^+)^+ = (xy)^+(xz)^+\]where the final equality holds by the defining Ehresmann identity $(ab)^+ = (ab^+)^+$. Thus $S$ satisfies the left $h$ condition. The right $h$ condition is seen dually and hence $S$ is $h$-Ehresmann. \qedhere
    
\end{proof}

\subsection{Examples of \texorpdfstring{$h$}{h}-Ehresmann semigroups}\label{sec:examples}

The class of $h$-Ehresmann semigroups contains many well-studied classes: any reduced Ehresmann monoid is $h$-Ehresmann, for example. Perhaps less trivially, any \textit{Sch\"utzenberger product} of right cancellative semigroups is left $h$-adequate \cite[Proposition 3.5]{fountain:munntype}. Fountain's \textit{left type B} semigroups \cite{fountain:adequate} are also such examples. 

Recall that a left Ehresmann semigroup $S$ is called \textit{restriction} if $xy^+ = (xy)^+x$ for all $x,y \in S$.
Restriction semigroups are left $h$-Ehresmann \cite[Lemma 2.9]{fountain:munntype} though not necessarily left $h$-adequate. Hence the class of left ample semigroups and the widely studied class of inverse semigroups are also examples. The latter two classes also form examples of left $h$-adequate semigroups. However there do exist left $h$-Ehresmann semigroups which are not inverse, left ample nor restriction. The following such example was observed by Aird and the author \cite{aird:growth} and served as the catalyst for this work.

\begin{prop}[{\cite[Lemma 3.3]{aird:growth}}]\label{prop:freemono}
    The free left Ehresmann monoid of rank $1$ is left $h$-adequate.
\end{prop}

It follows from Proposition \ref{prop:freemono} that the free left $h$-adequate monoid of rank $1$ is exactly the free left Ehresmann monoid of rank $1$. In higher ranks, this is not the case, as free left Ehresmann monoids of rank at least $2$ fail the left $h$ condition \cite[Proposition 3.2.31]{heath:graph}. However, we observe the following (perhaps surprising) fact that they do satisfy the right $h$ condition, using the left $X$-trees of Kambites \cite{kambites:freeleft}.

\begin{thm}\label{thm:flehisrhad}
    Free left Ehresmann monoids are right $h$-adequate.
\end{thm}

\begin{proof}
    The free left Ehresmann monoid of rank $0$ is the trivial $(2,1,0)$-algebra which is certainly right $h$-adequate. Now fix some non-empty set $X$. Consider any pruned left $X$-tree $\Gamma \in \fleh(X)$. If $\Gamma$ is an idempotent, define $\Gamma^\ast := \Gamma$. Otherwise, we may decompose $\Gamma$ as a gluing product $\Delta \times w \times T$ for unique left $X$-trees $\Delta \in \fleh(X)$, $w$ a word over $X$ (considered as a left $X$-tree consisting of simply a trunk labelled $w$), and an idempotent $T \in \fleh(X)$. Then define $\Gamma^\ast = T$. It is shown in \cite[Proposition 6.5, Corollary 6.6]{gould:coherent} that defining the operation ${}^\ast$ this way imbues $\fleh(X)$ with an adequate (in particular a right adequate) structure.

    To complete our proof, we need show the right $h$ condition. Note from above that the ${}^\ast$ operation on $\fleh(X)$ is defined to be \[\Gamma^\ast = \begin{cases}
        \Gamma & \text{if }\Gamma\text{ is idempotent,}\\
        T & \text{if }\Gamma\text{ is not idempotent where } \Gamma = \Delta \times w \times T \text{ as above.}
    \end{cases}\]Let $x,y,z \in \fleh(X)$ be given. We split into the following two cases.
    \begin{description}
        \item[Case 1] Suppose $y$ has no trunk edges. Then $y^\ast = y$. Hence \[(xz^\ast y)^\ast = (x^\ast z^\ast y^\ast)^\ast = x^\ast z^\ast y^\ast\]where the first equality follows from the right Ehresmann identity $(a^\ast b)^\ast = (ab)^\ast$. Similarly\[(xy)^\ast(zy)^\ast = (x^\ast y^\ast)^\ast(z^\ast y^\ast)^\ast = x^\ast y^\ast z^\ast y^\ast = x^\ast z^\ast y^\ast\]and the right $h$ condition holds.

        \item[Case 2] Suppose $y$ has at least one trunk edge. Since the tree $xz^\ast y$ has at least one trunk edge, we have by definition of ${}^\ast$ that $(xz^\ast y)^\ast = y^\ast$. Similarly, since both $xy$ and $zy$ have at least one trunk edge, we have $(xy)^\ast = y^\ast$ and $(zy)^\ast = y^\ast$. Thus $(xy)^\ast(zy)^\ast = y^\ast y^\ast = y^\ast$ and the right $h$ condition holds. \qedhere
    \end{description}
        
\end{proof}

Similar to the general Ehresmann case, (left, right) $h$-Ehresmann semigroups (monoids) form a variety. Thus we are enabled to study free objects in our context. We remark here that the free left Ehresmann monoid on $X$ being right $h$-Ehresmann certainly does not imply that it is the \textit{free} right $h$-Ehresmann monoid on $X$ -- note the differing signature.

Free right $h$-Ehresmann semigroups and monoids in all ranks have been previously described by Fountain \cite{fountain:h} in terms of certain Rees quotients of free products. The framework of $X$-trees set out in Section \ref{sec:trees} allows for much easier visualisation of free objects for the classes of Ehresmann semigroups, see for example the proof of Theorem \ref{thm:flehisrhad}. Our goal is to provide an alternative description in this manner for free left and right $h$-Ehresmann semigroups and monoids to aid their study, in particular one which recovers Kambites' description in the monogenic case. In rank $0$, it is clear that the free left/right $h$-Ehresmann monoid is the trivial $(2,1,0)$-algebra, so we mainly concern ourselves with ranks at least $1$.

In our aim to describe these, we begin with the following result mirroring that for free Ehresmann semigroups \cite[Proposition 2.2]{kambites:free}, which allows us to focus on the monoid case. 

\begin{prop}\label{prop:adjoiningidentity}
    Let $X$ be a set.
    \begin{enumerate}
        \item Let $F(X)$ be the free left \emph{[}resp. right\emph{]} $h$-Ehresmann semigroup on $X$. The free left \emph{[}resp. right\emph{]} $h$-Ehresmann monoid on $X$ is isomorphic to $F(X)$ with an adjoined element $1$ which acts as a multiplicative identity and has $1^+ := 1$ \emph{[}resp. $1^\ast:= 1$\emph{]}.
        \item Let $A(X)$ be the free left \emph{[}resp. right\emph{]} $h$-adequate semigroup on $X$. The free left \emph{[}resp. right\emph{]} $h$-adequate monoid on $X$ is isomorphic to $A(X)$ with an adjoined element $1$ which acts as a multiplicative identity and has $1^+ := 1$ \emph{[}resp. $1^\ast:= 1$\emph{]}.
    \end{enumerate}
    
\end{prop}

\begin{proof}
    We prove here the statement for free left $h$-Ehresmann semigroups, with the statement for free right $h$-Ehresmann semigroups and for free left/right $h$-adequate semigroups following similarly.

    Define $L(X) := F(X) \sqcup \{1\}$ with the operations as above. We show the $(2,1,0)$-algebra $L(X)$ is free on $X$. Let $M$ be any left $h$-Ehresmann monoid with identity $1_M$ and suppose $\chi: X \to M$ is a function. Since $M$ is, in particular, a left $h$-Ehresmann semigroup, there is a unique $(2,1)$-morphism $\chi^\ast:F(X) \to M$ extending $\chi$. Define the function $f: L(X) \to M$ by\[m \mapsto\begin{cases}
        \chi^\ast(m) &\text{if }m \in F(X),\\
        1_M &\text{if }m = 1.
    \end{cases}\]It is easily verified that $f$ is a $(2,1,0)$-morphism extending $\chi$. It remains to show that $f$ is the unique such morphism. Indeed any other $(2,1,0)$-morphism $g: L(X) \to M$ extending $\chi$ must restrict to a $(2,1)$-morphism $F(X) \to M$ and by uniqueness, this restriction must be $\chi^\ast$. Moreover $g(1) = 1_M$, and thus we must have $f = g$.

    Hence $L(X)$ is free on $X$. By uniqueness of free objects then, $L(X)$ is necessarily the free left $h$-Ehresmann monoid on $X$.
\end{proof}

\section{Splaying of directed trees}\label{sec:splaying}

Throughout this section, $X$ should be assumed to be a fixed non-empty set. We aim to describe the free left $h$-Ehresmann monoid on $X$.

Consider a left $X$-tree $\Gamma \in \lut(X)$ and suppose that the trunk edges of $\Gamma$ have the labels $x_1,x_2,\dots, x_k \in X$. It is clear by the form of left $X$-trees that there exist unique left $X$-trees $B_0,B_1,\dots,B_k \in \e(X)$, such that $\Gamma$ is the gluing product \[\Gamma = B_0 \times x_1 \times B_1 \times \dots \times x_k \times B_k.\]We call $B_i$ the $i\textsuperscript{th}$ \textit{branch} of $\Gamma$, and call the decomposition above the \textit{branch decomposition}. Figure \ref{fig:branches} shows an example of a tree $\Gamma$ and its three branches. Our framework will require defining a new operation on $X$-trees and will use this branch decomposition. 

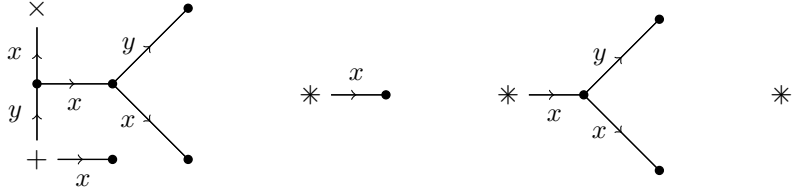
\begin{figure}[ht]
    \centering
    \begin{tikzpicture}
            \setup

            \node (C1) at (0,1) {\large$+$};
            \Vertex[x=1,y=1]{B1};
            \Vertex[x=0,y=2]{C2};
            \Vertex[x=1,y=2]{B2a};
            \Vertex[x=2,y=3]{B2aa};
            \Vertex[x=2,y=1]{B2ab};
            \node (C3) at (0,3) {\large$\times$};
        
            \Edge(C1)(B1)\draw (C1) -- (B1) node[midway,below=2pt] {$x$};
            \Edge(C1)(C2)\draw (C1) -- (C2) node[midway,left=2pt] {$y$};
            \Edge(C2)(B2a)\draw (C2) -- (B2a) node[midway,below=2pt] {$x$};
            \Edge(B2a)(B2aa)\draw (B2a) -- (B2aa) node[midway,left=2pt] {$y$};
            \Edge(B2a)(B2ab)\draw (B2a) -- (B2ab) node[midway,left=2pt] {$x$};
            \Edge(C2)(C3)\draw (C2) -- (C3) node[midway, left=2pt] {$x$};
        \end{tikzpicture}
        \hspace{3em}
        \begin{tikzpicture}
            \setup
            \node (_) at (0,0) {};
            \node (A) at (0,1.2) {\large$\PlusCross$};
            \Vertex[x=1,y=1.2]{U};
            \Edge(A)(U)\draw (A) -- (U) node[midway,above=2pt] {$x$};
        \end{tikzpicture}
        \hspace{3em}
        \begin{tikzpicture}
            \setup
            \node (_) at (0,0) {};
            \node (C2) at (0,1.2) {\large$\PlusCross$};
            \Vertex[x=1,y=1.2]{B2a};
            \Vertex[x=2,y=2.2]{B2aa};
            \Vertex[x=2,y=0.2]{B2ab};
        
            \Edge(C2)(B2a)\draw (C2) -- (B2a) node[midway,below=2pt] {$x$};
            \Edge(B2a)(B2aa)\draw (B2a) -- (B2aa) node[midway,left=2pt] {$y$};
            \Edge(B2a)(B2ab)\draw (B2a) -- (B2ab) node[midway,left=2pt] {$x$};
        \end{tikzpicture}
        \hspace{3em}
        \begin{tikzpicture}
            \setup
            \node (_) at (0,0) {};
            \node (pc) at (0,1.2) {\large$\PlusCross$};
        \end{tikzpicture}
        
    \caption{A left $\{x,y\}$-tree $\Gamma$ and its branches $B_0$, $B_1$ and $B_2$. Note that $B_2$ is the trivial, one-vertex tree. The branch decomposition of $\Gamma$ is $B_0 \times y \times B_1 \times x \times B_2$.}
    \label{fig:branches}
\end{figure}

\subsection{Splaying}

Consider the left $h$ condition $(xz^+y)^+ = (xy)^+(xz)^+$, say for elements $x,y,z \in X$. The left $X$-trees corresponding to both sides of this equation are seen in Figure \ref{fig:lefth}. Any description of free left $h$-Ehresmann monoids akin to that for Ehresmann monoids is required to identify these trees. 

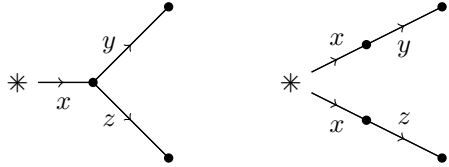
\begin{figure}[ht]
    \centering
    \begin{tikzpicture}
        \setup
        
        \node (C2) at (0,1) {\large$\PlusCross$};
        \Vertex[x=1,y=1]{B2a};
        \Vertex[x=2,y=2]{B2aa};
        \Vertex[x=2,y=0]{B2ab};
    
        \Edge(C2)(B2a)\draw (C2) -- (B2a) node[midway,below=2pt] {$x$};
        \Edge(B2a)(B2aa)\draw (B2a) -- (B2aa) node[midway,left=2pt] {$y$};
        \Edge(B2a)(B2ab)\draw (B2a) -- (B2ab) node[midway,left=2pt] {$z$};
    \end{tikzpicture}
    \hspace{3em}
    \begin{tikzpicture}
        \setup

        \node (C2) at (0,1) {\large$\PlusCross$};
        \Vertex[x=1,y=1.5]{B2ax};
        \Vertex[x=1,y=0.5]{B2ay};
        \Vertex[x=2,y=2]{B2aa};
        \Vertex[x=2,y=0]{B2ab};
    
        \Edge(C2)(B2ax)\draw (C2) -- (B2ax) node[midway,above=2pt] {$x$};
        \Edge(C2)(B2ay)\draw (C2) -- (B2ay) node[midway,below=2pt] {$x$};
        \Edge(B2ax)(B2aa)\draw (B2ax) -- (B2aa) node[midway,below=2pt] {$y$};
        \Edge(B2ay)(B2ab)\draw (B2ay) -- (B2ab) node[midway,above=2pt] {$z$};
    \end{tikzpicture}
    \caption{The left $X$-trees corresponding to the left $h$ condition.}
    \label{fig:lefth}
\end{figure}

This motivates a new operation which we call \textit{splaying}. Let $B$ be a left $X$-tree. We call a vertex $v$ of $B$ a \textit{leaf} if there are no edges with initial vertex $v$. Recall that for any vertex $v$ of $B$, there is a unique (possibly empty) directed path from the start vertex to $v$. Call this path $p_v$. We consider $p_v$ as a left $X$-tree in $\e(X)$ itself, with coinciding start and end vertex in the location induced by the start vertex of $B$. Given a tree $B \in \e(X)$, we define a left $X$-tree called the \textit{splaying} of $B$ to be\[\splay(B) := \bigtimes_{\substack{v\text{ is a}\\{\text{leaf of }B}}} p_v.\]We extend this operation to all left $X$-trees by using the branch decomposition: if $\Gamma \in \lut(X)$ has branch decomposition $B_0 \times x_1 \times B_1 \times \dots \times x_k \times B_k$, we define the splaying of $\Gamma$ to be\[\splay(\Gamma) := \splay(B_0) \times x_1 \times \splay(B_1) \times \dots \times x_k \times \splay(B_k).\]

Recall the example from Figure \ref{fig:branches}. Of the three branches of $\Gamma$, the branches $B_0$ and $B_2$ have $\splay(B_0) = B_0$ and $\splay(B_2) = B_2$. The tree $\splay(B_1) \neq B_1$; in fact $\splay(B_1)$ is the leftmost tree in Figure \ref{fig:splay}. The rightmost tree in Figure \ref{fig:splay} is the tree $\splay(\Gamma)$.

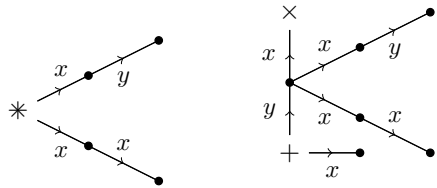
\begin{figure}[h]
    \centering
    \scalebox{.93}{
    \begin{tikzpicture}
            \setup

            \node (C2) at (0,2) {\large$\PlusCross$};
            \Vertex[x=1,y=2.5]{B2ax};
            \Vertex[x=1,y=1.5]{B2ay};
            \Vertex[x=2,y=3]{B2aa};
            \Vertex[x=2,y=1]{B2ab};
        
            \Edge(C2)(B2ax)\draw (C2) -- (B2ax) node[midway,above=2pt] {$x$};
            \Edge(C2)(B2ay)\draw (C2) -- (B2ay) node[midway,below=2pt] {$x$};
            \Edge(B2ax)(B2aa)\draw (B2ax) -- (B2aa) node[midway,below=2pt] {$y$};
            \Edge(B2ay)(B2ab)\draw (B2ay) -- (B2ab) node[midway,above=2pt] {$x$};
        \end{tikzpicture}
        \hspace{3em}
        \begin{tikzpicture}
            \setup

            \node (C1) at (0,1) {\large$+$};
            \Vertex[x=1,y=1]{B1};
            \Vertex[x=0,y=2]{C2};
            \Vertex[x=1,y=2.5]{B2ax};
            \Vertex[x=1,y=1.5]{B2ay};
            \Vertex[x=2,y=3]{B2aa};
            \Vertex[x=2,y=1]{B2ab};
            \node (C3) at (0,3) {\large$\times$};
        
            \Edge(C1)(B1)\draw (C1) -- (B1) node[midway,below=2pt] {$x$};
            \Edge(C1)(C2)\draw (C1) -- (C2) node[midway,left=2pt] {$y$};
            \Edge(C2)(B2ax)\draw (C2) -- (B2ax) node[midway,above=2pt] {$x$};
            \Edge(C2)(B2ay)\draw (C2) -- (B2ay) node[midway,below=2pt] {$x$};
            \Edge(B2ax)(B2aa)\draw (B2ax) -- (B2aa) node[midway,below=2pt] {$y$};
            \Edge(B2ay)(B2ab)\draw (B2ay) -- (B2ab) node[midway,above=2pt] {$x$};
            \Edge(C2)(C3)\draw (C2) -- (C3) node[midway, left=2pt] {$x$};
        \end{tikzpicture}
        }
    \caption{The trees $\splay(B_1)$ and $\splay(\Gamma)$ for the $\Gamma$ in Figure \ref{fig:branches}.}
    \label{fig:splay}
\end{figure}

We consider here the interaction of $\splay$ with our previously defined operations $\times$ and ${}^{(+)}$ on $\lut(X)$.

\begin{lem}\label{lem:splay}
    Let $\Gamma,\Delta \in \lut(X)$. Then:
    \begin{enumerate}
        \item $\splay(\Gamma \times \Delta) = \splay(\Gamma) \times \splay(\Delta)$.\label{lem:splay:splaymor}
        \item $\splay\left(\overline{\splay(\Gamma)}\right) = \overline{\splay(\Gamma)}$.\label{lem:splay:splaybar}
        \item $\splay\left(\Gamma^{(+)}\right) = \splay\left(\splay(\Gamma)^{(+)}\right)$.\label{lem:splay:splaygamma+}
        \item Suppose $\rho$ is a retraction on $\Gamma$. Then $\overline{\splay\left( \Gamma \right)} = \overline{\splay(\rho(\Gamma))}$.\label{lem:splay:splayrhogamma}
    \end{enumerate}
\end{lem}

\begin{proof}
    Throughout, suppose $\Gamma$ has branch decomposition $B_0 \times x_1 \times B_1 \times \dots \times x_k \times B_k$.
    \begin{enumerate}
        \item Suppose $\Delta$ has branch decomposition $C_0 \times y_1 \times C_1 \times \dots \times y_l \times C_l$. It is clear that the branch decomposition of $\Gamma \times \Delta$ is exactly \[B_0 \times x_1 \times B_1 \times \dots \times x_k \times (B_k \times C_0) \times y_1 \times C_1 \times \dots \times y_l \times C_l.\]Note that \[\splay(B_k \times C_0) = \bigtimes_{\substack{v\text{ is a}\\\substack{\text{leaf of}\\B_k \times C_0}}} p_v = \bigtimes_{\substack{v\text{ is a}\\\substack{\text{leaf of}\\B_k}}} p_v \times \bigtimes_{\substack{v\text{ is a}\\\substack{\text{leaf of}\\C_0}}} p_v = \splay(B_k) \times \splay(C_0).\]Hence certainly \begin{multline*}
            \splay(\Gamma \times \Delta) = \splay(B_0) \times x_1 \times \splay(B_1) \times \dots \times x_k \times \splay(B_k) \times\\ \splay(C_0) \times y_1 \times \splay(C_1) \times \dots \times y_l \times \splay(C_l) = \splay(\Gamma) \times \splay(\Delta).
        \end{multline*}

        \item Certainly $\overline{\splay(\Gamma)}$ exists as some subtree of $\splay(\Gamma)$ containing the start and end vertices. In particular, the branch decomposition of $\overline{\splay(\Gamma)}$ is of the form \[D_0 \times x_1 \times D_1 \times \dots \times x_k \times D_k\]where each $D_i$ is a subtree of $\splay(B_i)$. Clearly then $\splay(D_i) = D_i$ for each $i$, and as such \begin{align*}
            \splay\left(\overline{\splay(\Gamma)}\right) &= \splay(D_0) \times x_1 \times \splay(D_1) \times \dots \times x_k \times \splay(D_k)\\
            &= D_0 \times x_1 \times D_1 \times \dots \times x_k \times D_k\\
            &= \overline{\splay(\Gamma)}.
        \end{align*}

        \item First note that $\splay(\Gamma^{(+)})$ and $\splay(\splay(\Gamma)^{(+)})$ are both trees in $\e(X)$ in the image of $\splay$. Hence they are entirely determined by the labels of paths from their start vertex which terminate at leaves. It is therefore sufficient to show that a path exists from the start vertex of $\splay(\Gamma^{(+)})$ to a leaf if and only if a path with the same label exists from the start vertex of $\splay(\splay(\Gamma)^{(+)})$ to a leaf.
        
        Suppose that a path $p$ exists in $\splay(\Gamma^{(+)})$ from the start vertex to some leaf. By definition of $\splay$, this path exists if and only if a path with the same label exists in $\Gamma^{(+)}$ from the start vertex to some leaf. By definition of ${}^{(+)}$, this path exists if and only if a corresponding path exists in $\Gamma$. In turn, this is equivalent to a corresponding path existing in $\splay(\Gamma)$, in $\splay(\Gamma)^{(+)}$ and in $\splay(\splay(\Gamma)^{(+)})$. Our result is proved. 

        \item For the proof of \eqref{lem:splay:splayrhogamma}, we introduce a new operation on left $X$-trees which we will not require elsewhere. We call this operation \textit{full splaying}. Given a tree $B \in \e(X)$, we define \[\fullsplay(B) := \bigtimes_{v \in V(B)} p_v.\]where $p_v$ is the unique directed path from the start vertex of $B$ to $v$, considered as a left $X$-tree with coinciding start and end vertex in the expected place. Note that $\splay(B)$ is a subtree of $\fullsplay(B)$ and moreover $\overline{\fullsplay(B)} = \overline{\splay(B)}$. As with our regular splaying operation, we extend $\fullsplay$ to all of $\lut(X)$ by defining\[\fullsplay(\Xi) := \fullsplay(E_0) \times z_1 \times \fullsplay(E_1) \times \dots \times z_m \times \fullsplay(E_m)\]for all $\Xi \in \lut(X)$ where $E_0 \times z_1 \times E_1 \times \dots \times z_m \times E_m$ is the branch decomposition of $\Xi$. As above, we observe that $\splay(\Xi)$ is a subtree of $\fullsplay(\Xi)$ and $\overline{\fullsplay(\Xi)} = \overline{\splay(\Xi)}$.

        Now consider the trees $\fullsplay(\Gamma)$ and $\fullsplay(\rho(\Gamma))$. Since $\rho(\Gamma)$ is a subtree of $\Gamma$, it follows that $\fullsplay(\rho(\Gamma))$ is a subtree of $\fullsplay(\Gamma)$. We aim to construct a retraction $\phi:\fullsplay(\Gamma) \to \fullsplay(\Gamma)$ with image $\fullsplay(\rho(\Gamma))$. Figure \ref{fig:rhodiagram} gives an example of a tree $\Gamma$, an image of $\Gamma$ under a retraction $\rho$, and the respective full splaying images.
        
        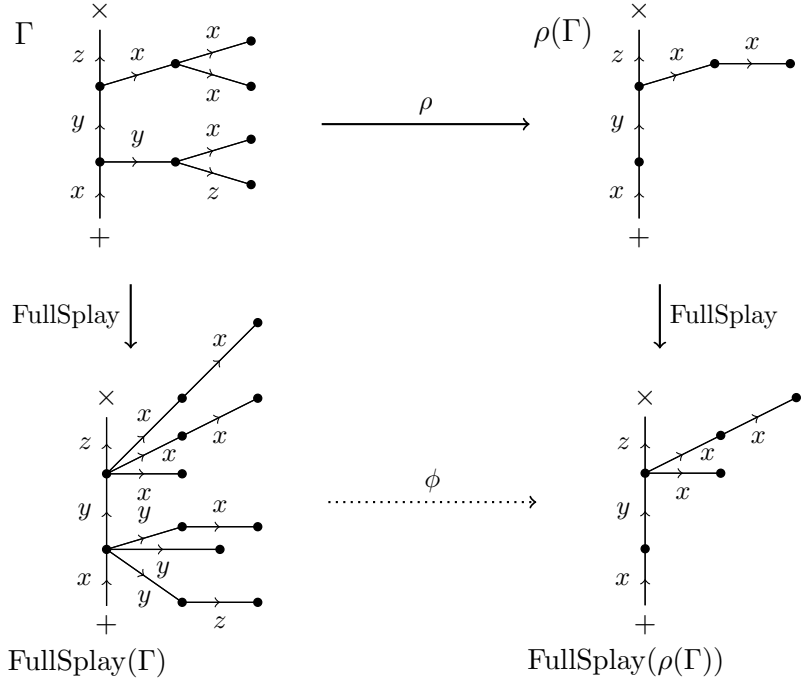
\begin{figure}[h!]
            \centering
            \begin{tikzpicture}

            \node (upleft) at (0,0) {
            \begin{tikzpicture}
                \setup

                \node (l) at (-1,2.7) {\Large$\Gamma$};
                \node (A) at (0,0) {\large$+$};
                \Vertex[x=0,y=1]{C1};
                \Vertex[x=0,y=2]{C2};
                \Vertex[x=1,y=1]{B1};
                \Vertex[x=2,y=1.3]{BA};
                \Vertex[x=2,y=0.7]{BB};
                \Vertex[x=1,y=2.3]{E1};
                \Vertex[x=2,y=2.6]{EA};
                \Vertex[x=2,y=2]{EB};
                \node (C3) at (0,3) {\large$\times$};
            
                \Edge(A)(C1)\draw (A) -- (C1) node[midway,left=2pt] {$x$};
                \Edge(C1)(C2)\draw (C1) -- (C2) node[midway,left=2pt] {$y$};
                \Edge(C2)(C3)\draw (C2) -- (C3) node[midway, left=2pt] {$z$};
                \Edge(C1)(B1)\draw (C1) -- (B1) node[midway,above=2pt] {$y$};
                \Edge(B1)(BA)\draw (B1) -- (BA) node[midway,above=2pt] {$x$};
                \Edge(B1)(BB)\draw (B1) -- (BB) node[midway,below=2pt] {$z$};
                \Edge(C2)(E1)\draw (C2) -- (E1) node[midway,above=2pt] {$x$};
                \Edge(E1)(EA)\draw (E1) -- (EA) node[midway,above=2pt] {$x$};
                \Edge(E1)(EB)\draw (E1) -- (EB) node[midway,below=2pt] {$x$};
            \end{tikzpicture}
            };
            
            \node (upright) at (7,0) {
            \begin{tikzpicture}
                \setup
            
                \node (l) at (-1,2.7) {\Large$\rho(\Gamma)$};
                \node (A) at (0,0) {\large$+$};
                \Vertex[x=0,y=1]{C1};
                \Vertex[x=0,y=2]{C2};
                \Vertex[x=1,y=2.3]{E1};
                \Vertex[x=2,y=2.3]{EA};

                \node (C3) at (0,3) {\large$\times$};
            
                \Edge(A)(C1)\draw (A) -- (C1) node[midway,left=2pt] {$x$};
                \Edge(C1)(C2)\draw (C1) -- (C2) node[midway,left=2pt] {$y$};
                \Edge(C2)(C3)\draw (C2) -- (C3) node[midway, left=2pt] {$z$};
                \Edge(C2)(E1)\draw (C2) -- (E1) node[midway,above=2pt] {$x$};
                \Edge(E1)(EA)\draw (E1) -- (EA) node[midway,above=2pt] {$x$};
            \end{tikzpicture}
            };

            \node (downleft) at (0,-5) {
            \begin{tikzpicture}
                \setup

                \node (l) at (-0.25,-0.5) {\large$\fullsplay(\Gamma)$};
                \node (A) at (0,0) {\large$+$};
                \Vertex[x=0,y=1]{C1};
                \Vertex[x=0,y=2]{C2};
                \Vertex[x=1,y=0.3]{B1};
                \Vertex[x=1,y=1.3]{B2};
                \Vertex[x=1.5,y=1]{B3};
                \Vertex[x=2,y=1.3]{BA};
                \Vertex[x=2,y=0.3]{BB};
                \Vertex[x=1,y=2.5]{E1};
                \Vertex[x=1,y=3]{E2};
                \Vertex[x=1,y=2]{E3};
                \Vertex[x=2,y=4]{EA};
                \Vertex[x=2,y=3]{EB};
                \node (C3) at (0,3) {\large$\times$};
            
                \Edge(A)(C1)\draw (A) -- (C1) node[midway,left=2pt] {$x$};
                \Edge(C1)(C2)\draw (C1) -- (C2) node[midway,left=2pt] {$y$};
                \Edge(C2)(C3)\draw (C2) -- (C3) node[midway, left=2pt] {$z$};
                \Edge(C1)(B1)\draw (C1) -- (B1) node[midway,below=2pt] {$y$};
                \Edge(C1)(B2)\draw (C1) -- (B2) node[midway,above=2pt] {$y$};
                \Edge(C1)(B3)\draw (C1) -- (B3) node[midway,below=1pt] {$y$};
                \Edge(B1)(BB)\draw (B1) -- (BB) node[midway,below=2pt] {$z$};
                \Edge(B2)(BA)\draw (B2) -- (BA) node[midway,above=2pt] {$x$};
                \Edge(C2)(E1)\draw (C2) -- (E1) node[midway,right=3pt] {$x$};
                \Edge(C2)(E2)\draw (C2) -- (E2) node[midway,above=2pt] {$x$};
                \Edge(C2)(E3)\draw (C2) -- (E3) node[midway,below=2pt] {$x$};
                \Edge(E2)(EA)\draw (E2) -- (EA) node[midway,above=2pt] {$x$};
                \Edge(E1)(EB)\draw (E1) -- (EB) node[midway,below=2pt] {$x$};
            \end{tikzpicture}
            };

            \node (downright) at (7,-5.4) {
            \begin{tikzpicture}
                \setup

                \node (l) at (-0.25,-0.5) {\large$\fullsplay(\rho(\Gamma))$};
                \node (A) at (0,0) {\large$+$};
                \Vertex[x=0,y=1]{C1};
                \Vertex[x=0,y=2]{C2};
                \Vertex[x=1,y=2.5]{E1};
                \Vertex[x=1,y=2]{E3};
                \Vertex[x=2,y=3]{EB};
                \node (C3) at (0,3) {\large$\times$};
            
                \Edge(A)(C1)\draw (A) -- (C1) node[midway,left=2pt] {$x$};
                \Edge(C1)(C2)\draw (C1) -- (C2) node[midway,left=2pt] {$y$};
                \Edge(C2)(C3)\draw (C2) -- (C3) node[midway, left=2pt] {$z$};
                \Edge(C2)(E1)\draw (C2) -- (E1) node[midway,right=3pt] {$x$};
                \Edge(C2)(E3)\draw (C2) -- (E3) node[midway,below=2pt] {$x$};
                \Edge(E1)(EB)\draw (E1) -- (EB) node[midway,below=2pt] {$x$};
            \end{tikzpicture}
            };
            
            \node[right=0.5cm of upleft] (larrow1) {};
            \node[right=3.5cm of upleft] (rarrow1) {};
            \draw[->,thick] (larrow1) -- node[above] {$\rho$} (rarrow1);

            \node[below=0cm of upleft] (larrow2) {};
            \node[below=1.1cm of upleft] (rarrow2) {};
            \draw[->,thick] (larrow2) -- node[left] {$\fullsplay$} (rarrow2);

            \node[below=0cm of upright] (larrow3) {};
            \node[below=1.1cm of upright] (rarrow3) {};
            \draw[->,thick] (larrow3) -- node[right] {$\fullsplay$} (rarrow3);

            \node[right=0.5cm of downleft] (larrow4) {};
            \node[right=3.5cm of downleft] (rarrow4) {};
            \draw[->,thick,dotted] (larrow4) -- node[above] {$\phi$} (rarrow4);
            
            \end{tikzpicture}
            \caption{Example of the maps in Lemma \ref{lem:splay}\eqref{lem:splay:splayrhogamma} for a left $\{x,y,z\}$-tree $\Gamma$.}
            \label{fig:rhodiagram}
        \end{figure}

        We first define the map $\phi$ on vertices of $\fullsplay(\Gamma)$. We define $\phi$ to fix each trunk vertex. Now consider a leaf vertex not on the trunk, say $v$. Then there is a unique path in $\fullsplay(\Gamma)$, say labelled $w$, from a trunk vertex (say the $i^\text{th}$ trunk vertex) to $v$. By definition, $v$ corresponds to some directed path $p_v$ in $\Gamma$ from the $i^\text{th}$ trunk vertex $t_v$ to some non-trunk vertex $u_v$ with label $w$. These vertices and path have images in $\rho(\Gamma)$, i.e. in $\rho(\Gamma)$ there is a directed path $\rho(p_v)$ from the $i^\text{th}$ trunk vertex $\rho(t_v)$ to some vertex $\rho(u_v)$. Now consider two cases of $\rho(u_v)$. \begin{description}
            \item[Case 1] If $\rho(u_v)$ is on the trunk of $\rho(\Gamma)$, say it is the $j^\text{th}$ trunk vertex, then $w$, the label of $\rho(p_v)$, must label of the trunk from the $i^\text{th}$ trunk vertex to the $j^\text{th}$ trunk vertex of $\Gamma$. Since the trunk is fixed under full splaying, this label persists from the $i^\text{th}$ to the $j^\text{th}$ trunk vertex of $\fullsplay(\Gamma)$. Thus define $\phi(v)$ to be the $j^\text{th}$ trunk vertex of $\fullsplay(\Gamma)$.
            \item[Case 2] If instead $\rho(u_v)$ is not on the trunk of $\rho(\Gamma)$, then it corresponds to some path in $\fullsplay(\rho(\Gamma))$ from a trunk vertex to a leaf. Suppose this trunk vertex is the $k^\text{th}$ trunk vertex, the leaf is $l$, and the path has label $b$. Then $w$, the label of $\rho(p_v)$, is of the form $ab$, where $a$ labels the trunk from the $i^\text{th}$ to the $k^\text{th}$ vertex. Considering $\fullsplay(\rho(\Gamma))$ as a subtree of $\fullsplay(\Gamma)$, define $\phi(v) = l$. Note the path from $v$ to $l$ in $\fullsplay(\Gamma)$ has label $ab = w$.
        \end{description}
        
        We now define $\phi$ on all other vertices. For any non-trunk, non-leaf vertex $u$ of $\fullsplay(\Gamma)$, which must lie on some single path $p_{v}$ from a trunk vertex $t$ to a leaf $v$, we define $\phi(u)$ to be the natural vertex on the path from $\phi(t)$ to $\phi(v)$. By choice of $\phi(v)$, this path has the same label as $p_{v}$ and so a natural choice for $\phi(u)$ exists.

        Since $\fullsplay(\Gamma)$ consists only of the trunk and non-splitting paths, the map $\phi$ on vertices induces a map on edges by defining the image of an edge from a vertex $p$ to a vertex $q$ to be the unique edge from $\phi(p)$ to $\phi(q)$, and with this extension, $\phi$ forms an endomorphism on $\fullsplay(\Gamma)$.

        Observe that for any vertex $v$ in the embedded copy of $\fullsplay(\rho(\Gamma)) \subseteq \fullsplay(\Gamma)$, the map $\phi$ must fix $v$ and consequently also any edges of $\fullsplay(\rho(\Gamma))$. Hence $\phi$ is a retraction of $\fullsplay(\Gamma)$ with image $\fullsplay(\rho(\Gamma))$. By Proposition \ref{prop:retractconfluent}, we deduce that $\overline{\fullsplay(\Gamma)} = \overline{\fullsplay(\rho(\Gamma))}$. Thus\[\overline{\splay(\Gamma)} = \overline{\fullsplay(\Gamma)} = \overline{\fullsplay(\rho(\Gamma))} = \overline{\splay(\rho(\Gamma))}.\qedhere\] 
    \end{enumerate}
\end{proof}

\subsection{The left \texorpdfstring{$h$}{h}-adequate monoid \texorpdfstring{$\h(X)$}{H(X)}} The splaying operation provides us with the tools required to describe our free objects. Given a left $X$-tree $\Gamma$, define $\widehat\Gamma$ to be \[\widehat{\Gamma} := \overline{\splay(\Gamma)}.\]

We define the set\[\h(X) = \left\{ \widehat{\Gamma} \mid \Gamma \in \lut(X) \right\}.\]Similar to retraction, we treat $\widehat{\cdot}$ as a map from $\lut(X)$ to $\h(X)$, and we imbue $\h(X)$ with a $(2,1,0)$-structure by defining operations\[\Gamma\Delta := \widehat{\Gamma \times \Delta},\quad\Gamma^+ := \widehat{\Gamma^{(+)}},\]and distinguishing the trivial tree. Note that the generating set $X$ of $\lut(X)$ is contained in $\h(X)$. We observe the following alternative form of the product in $\h(X)$, which simplifies much of the forthcoming.

\begin{lem}\label{lem:altproduct}
    Let $\Gamma,\Delta \in \h(X)$. Then $\Gamma\Delta = \overline{\Gamma \times \Delta}$.
\end{lem}

\begin{proof}
    Observe that if $\Gamma \in \h(X)$, say $\Gamma = \widehat{\Xi}$ for some $\Xi \in \lut(X)$, then\begin{equation}\label{eq:splaygammaisgamma}
        \splay(\Gamma) = \splay\left(\overline{\splay(\Xi)}\right) = \overline{\splay(\Xi)} = \Gamma
    \end{equation}by Lemma \ref{lem:splay}\eqref{lem:splay:splaybar}. Thus by Lemma \ref{lem:splay}\eqref{lem:splay:splaymor}, we have\[\Gamma\Delta = \overline{\splay(\Gamma \times \Delta)} = \overline{\splay(\Gamma) \times \splay(\Delta)} = \overline{\Gamma \times \Delta}. \qedhere\]
\end{proof}

Observe that by Lemma \ref{lem:splay}\eqref{lem:splay:splayrhogamma}, we have that $\overline{\splay\left(\overline{\Gamma}\right)} = \overline{\splay(\Gamma)}$ for any $\Gamma \in \lut(X)$. Hence if $\Gamma,\Delta \in \lut(X)$ have $\overline{\Gamma} = \overline{\Delta}$, then $\widehat{\Gamma} = \widehat{\overline{\Gamma}} = \widehat{\overline{\Delta}} = \widehat{\Delta}$. Thus $\widehat{\cdot}$ restricts to a well defined map $\fleh(X) \to \h(X)$ by $\Gamma \mapsto \widehat{\Gamma_0}$ where $\Gamma_0$ is any left $X$-tree such that $\overline{\Gamma_0} = \Gamma$. In particular, one can take $\Gamma_0 = \Gamma$ itself, and then this map is exactly $\Gamma \mapsto \widehat{\Gamma}$ which we also denote $\widehat{\cdot}$. We now show the following key result, mirroring \cite[Theorem 4.5]{kambites:free}.

\begin{thm}\label{thm:hatmor}
    The map $\widehat{\cdot} : \fleh(X) \to \h(X)$ is a surjective $(2,1,0)$-morphism fixing $X$.
\end{thm}

\begin{proof}
    Certainly $\widehat{\cdot}$ fixes $X$. We need to verify surjectivity and that \[\widehat\Gamma \widehat\Delta = \widehat{\overline{\Gamma\times\Delta}}\text{ ,}\quad \left(\widehat{\Gamma}\right)^+ = \widehat{\overline{\Gamma^{(+)}}}\mkern18mu\text{and}\mkern18mu \widehat{\varepsilon} = \varepsilon\]for all $\Gamma,\Delta \in \lut(X)$ and for the trivial tree $\varepsilon$. One establishes the final claim quickly, since $\overline{\splay(\varepsilon)} = \overline{\varepsilon} = \varepsilon$ indeed.

    For surjectivity, suppose $\widehat{\Gamma} \in \h(X)$ for some $\Gamma \in \lut(X)$. Clearly $\widehat{\Gamma}$ is a pruned left $X$-tree, with the property that\begin{equation}\label{eq:hathatishat}
        \widehat{\widehat{\Gamma}} = \overline{\splay\left(\widehat{\Gamma}\right)} = \overline{\widehat{\Gamma}} = \overline{\overline{\splay(\Gamma)}} = \overline{\splay(\Gamma)} = \widehat{\Gamma}
    \end{equation}where the second equality follows from \eqref{eq:splaygammaisgamma} and the fourth equality from Lemma \ref{lem:barismorphism}\eqref{lem:barismorphism:barbar}. Hence $\widehat{\Gamma}$ is its own image under $\widehat{\cdot}$ and the map is indeed surjective.

    We now show that $\widehat\Gamma \widehat\Delta = \widehat{\overline{\Gamma\times\Delta}}$. Via our previous lemmas,\begin{align*}
        \widehat{\Gamma}\widehat{\Delta} &= \overline{\widehat{\Gamma} \times \widehat{\Delta}} && \text{(by Lemma \ref{lem:altproduct})}\\
        &= \overline{\overline{\splay(\Gamma)} \times \overline{\splay(\Delta)}} && \\
        &= \overline{\splay(\Gamma) \times \splay(\Delta)} && \text{(by Lemma \ref{lem:barismorphism}\eqref{lem:barismorphism:bartimesbar})}\\
        &= \overline{\splay(\Gamma \times \Delta)} && \text{(by Lemma \ref{lem:splay}\eqref{lem:splay:splaymor})}\\
        &= \overline{\splay\left(\,\overline{\Gamma \times \Delta}\,\right)} && \text{(by Lemma \ref{lem:splay}\eqref{lem:splay:splayrhogamma})}\\
        &= \widehat{\overline{\Gamma \times \Delta}}.
    \end{align*}It remains to show that $\left(\widehat{\Gamma}\right)^+ = \widehat{\overline{\Gamma^{(+)}}}$. Note that\[\left(\widehat{\Gamma}\right)^+ = \widehat{\widehat{\Gamma}^{(+)}} = \overline{\splay\left( \widehat{\Gamma}^{(+)} \right)} = \overline{\splay\left( \overline{\splay(\Gamma)}^{(+)}\right)}.\]Consider a retraction $\rho$ on $\splay(\Gamma)$ with image $\overline{\splay(\Gamma)}$. Since the underlying graph of $\splay(\Gamma)^{(+)}$ is the same as $\splay(\Gamma)$, only with relocated end vertex, the map $\rho$ admits a natural retraction on $\splay(\Gamma)^{(+)}$, with the image of this retraction being $\overline{\splay(\Gamma)}^{(+)}$. Thus by Lemma \ref{lem:splay}\eqref{lem:splay:splayrhogamma}, we have $\overline{\splay\left( \overline{\splay(\Gamma)}^{(+)}\right)} = \overline{\splay\left( \splay(\Gamma)^{(+)}\right)}$. Furthermore, by Lemma \ref{lem:splay}\eqref{lem:splay:splaygamma+} and Lemma \ref{lem:splay}\eqref{lem:splay:splayrhogamma}, we see that $\overline{\splay\left( \splay(\Gamma)^{(+)}\right)} = \overline{\splay\left(\Gamma^{(+)} \right)} = \overline{\splay\left(\,\overline{\Gamma^{(+)}}\,\right)}$. Overall,\[\left(\widehat{\Gamma}\right)^+ = \overline{\splay\left( \overline{\splay(\Gamma)}^{(+)}\right)} = \overline{\splay\left(\,\overline{\Gamma^{(+)}}\,\right)} = \widehat{\overline{\Gamma^{(+)}}}.\qedhere\]    
\end{proof}

We quickly conclude the following.

\begin{cor}\label{cor:hisleftE}
    The $(2,1,0)$-algebra $\h(X)$ is a left Ehresmann monoid generated by $X$.
\end{cor}

\begin{proof}
    Follows from Theorem \ref{thm:hatmor} and Birkhoff's Theorem \cite[Theorem 10]{birkhoff:struct}.
\end{proof}

Left adequate monoids only form a quasi-variety and are not closed under taking $(2,1,0)$-quotients. As such, we cannot use Birkhoff's Theorem to deduce that $\h(X)$ is left adequate. Instead, we need to show this directly.

\begin{thm}
    The $(2,1,0)$-algebra $\h(X)$ is a left adequate monoid generated by $X$.
\end{thm}

\begin{proof}
    Following Corollary \ref{cor:hisleftE}, it remains to demonstrate the two remaining quasi-identities for left adequacy. For any tree $\Gamma \in \h(X)$, let $\Theta(\Gamma)$ be the subtree consisting exactly of the trunk of $\Gamma$. If $\Gamma^2 = \Gamma$, then $\overline{\Gamma \times \Gamma} = \Gamma$ by Lemma \ref{lem:altproduct}. Thus certainly $\Theta(\Gamma) \times \Theta(\Gamma) = \Theta(\Gamma)$, and so we must have that $\Theta(\Gamma)$ is the trivial tree. Hence $\Gamma$ is such that $\Gamma = \Gamma^{(+)}$, and it follows that $\Gamma = \Gamma^+$.

    Now let $\Delta,\Xi \in \h(X)$ also. Suppose $\Gamma\Xi = \Delta\Xi$, that is $\overline{\Gamma \times \Xi} = \overline{\Delta \times \Xi}$ by Lemma \ref{lem:altproduct}. Thus\begin{equation}
        \Gamma\Xi = \Delta\Xi \iff \overline{\Gamma \times \Xi} = \overline{\Delta \times \Xi} \iff \overline{\Gamma \times \Xi^{(+)}} = \overline{\Delta \times \Xi^{(+)}} \label{eq:GammaXi+}
    \end{equation}where the final equivalence is by Lemma \ref{lem:barismorphism}\eqref{lem:barismorphism:adequatequasi}. We now claim that \[\overline{\Gamma \times \Xi^{(+)}} = \overline{\Delta \times \Xi^{(+)}} \implies \overline{\Gamma \times \splay(\Xi^{(+)})} = \overline{\Delta \times \splay(\Xi^{(+)})}.\]
    
    Assume then that $\overline{\Gamma \times \Xi^{(+)}} = \overline{\Delta \times \Xi^{(+)}}$. In particular, $\Gamma$ and $\Delta$ must have identical trunks; suppose they both have $k$ trunk vertices and label them $1,\dots,k$. Since both $\overline{\Gamma \times \splay(\Xi^{(+)})}$ and $\overline{\Delta \times \splay(\Xi^{(+)})}$ are fixed by $\splay$, in order to show they are equal it is sufficient to show that a path exists in $\overline{\Gamma \times \splay(\Xi^{(+)})}$ from a trunk vertex to a leaf if and only if a path with the same label exists in $\overline{\Delta \times \splay(\Xi^{(+)})}$ from the corresponding trunk vertex to a leaf.

    Suppose that a path labelled $w$ exists in $\overline{\Gamma \times \splay(\Xi^{(+)})}$ from the $i^\text{th}$ trunk vertex to a leaf. We split into two cases.
    \begin{description}
        \item[Case 1] If $i \lneq k$, the path labelled $w$ exists in the $\Gamma$ portion of $\overline{\Gamma \times \splay(\Xi^{(+)})}$. Since this tree and $\Gamma$ are pruned, the label of $w$ must not be of the form $tp$ for $t$ the label of the remaining trunk of $\Gamma$ from the $i^\text{th}$ trunk vertex to the end vertex, and for any $p$ which labels a path in $\splay(\Xi^{(+)})$ from the start vertex. Since the labels of paths readable (in this sense) in $\splay(\Xi^{(+)})$ are exactly the labels of paths readable in $\Xi^{(+)}$, the path labelled $w$ (which exists in $\Gamma$) is fixed by every retraction on $\Gamma \times \Xi^{(+)}$ (as otherwise it would map to a path which we concluded above cannot exist). Since $\overline{\Gamma \times \Xi^{(+)}} = \overline{\Delta \times \Xi^{(+)}}$, a corresponding path labelled $w$ exists in $\overline{\Delta \times \Xi^{(+)}}$ from the $i^\text{th}$ trunk vertex of $\Delta$ to a leaf. Such a path thus exists in $\Delta$, and so too in $\Delta \times \splay(\Xi^{(+)})$. Moreover it must exist in $\overline{\Delta \times \splay(\Xi^{(+)})}$, as otherwise an illegal path as described above would exist.

        \item[Case 2] Suppose $i = k$. The path labelled $w$ either exists in $\Gamma$ with initial vertex the end vertex of $\Gamma$, or it must exist in $\splay(\Xi^{(+)})$ from the start vertex (or potentially both). If it exists in $\Gamma$ and not in $\splay(\Xi^{(+)})$, then an identical argument to Case 1 gives our claim. Alternatively, if it exists in $\splay(\Xi^{(+)})$, then it certainly exists in $\Delta \times \splay(\Xi^{(+)})$ with initial vertex the end vertex. We claim it exists in $\overline{\Delta \times \splay(\Xi^{(+)})}$ and terminates at a leaf. Suppose it does not: then some retraction on $\Delta \times \splay(\Xi^{(+)})$ must ``erase'' it, i.e. map the path to some path with label $wq$ in $\overline{\Delta \times \splay(\Xi^{(+)})}$ from the end vertex to a leaf.
        
        If the path labelled $wq$ exists in $\splay(\Xi^{(+)})$, then it certainly exists in $\Gamma \times \splay(\Xi^{(+)})$. But then there is some retraction mapping our path $w$ in $\overline{\Gamma \times \splay(\Xi^{(+)})}$ onto this path, so as to not contradict being pruned, we must have that $q$ is the empty word and there is a path labelled $w$ in $\overline{\Delta \times \splay(\Xi^{(+)})}$ with terminal vertex a leaf as required.
        
        Instead, if the path labelled $wq$ exists in $\Delta$, then some path labelled $wq$ is readable in $\overline{\Delta \times \Xi^{(+)}} = \overline{\Gamma \times \Xi^{(+)}}$ with initial vertex the end vertex. Then it is either readable in $\Gamma$ or in $\Xi^{(+)}$ (equivalently in $\splay(\Xi^{(+)})$). We have already dealt with the latter case above, and if it exists in $\Gamma$ then it exists in $\Gamma \times \splay(\Xi^{(+)})$ and hence in $\overline{\Gamma \times \splay(\Xi^{(+)})}$. But again the original path labelled $w$ also exists here, and so as to not contradict being pruned, we must have that $q$ is the empty word and thus there is a path labelled $w$ with terminal vertex a leaf in $\overline{\Delta \times \splay(\Xi^{(+)})}$.
    \end{description}

    A symmetric argument further gives that $\overline{\Gamma \times \splay(\Xi^{(+)})}$ and $\overline{\Delta \times \splay(\Xi^{(+)})}$ have exactly the same branches and hence \[\overline{\Gamma \times \splay(\Xi^{(+)})} = \overline{\Delta \times \splay(\Xi^{(+)})}.\]This completes the proof of our claim.
    
    Hence by \eqref{eq:GammaXi+}, we have that $\Gamma\Xi = \Delta\Xi$ implies $\overline{\Gamma \times \splay(\Xi^{(+)})} = \overline{\Delta \times \splay(\Xi^{(+)})}$. By Lemma \ref{lem:barismorphism}\eqref{lem:barismorphism:bartimesbar} and Lemma \ref{lem:altproduct}, observe that \[\overline{\Gamma \times \splay(\Xi^{(+)})} = \overline{\Gamma \times \overline{\splay(\Xi^{(+)})}} = \overline{\Gamma \times \Xi^+} = \Gamma\Xi^+.\] Similarly $\overline{\Delta \times \splay(\Xi^{(+)})} = \Delta\Xi^+$ and thus \[\Gamma\Xi^+ = \overline{\Gamma \times \splay(\Xi^{(+)})} = \overline{\Delta \times \splay(\Xi^{(+)})} = \Delta\Xi^+. \qedhere\]

\end{proof}

\begin{thm}
    The $(2,1,0)$-algebra $\h(X)$ satisfies the left $h$ identity.
\end{thm}

\begin{proof}
    Let $\Gamma,\Delta,\Xi \in \h(X)$ and consider the elements $(\Gamma\Xi^+\Delta)^+$ and $(\Gamma\Delta)^+(\Gamma\Xi)^+ = (\Gamma\Delta)^+(\Gamma\Xi^+)^+$. Both of these elements are fixed by $\splay$ by \eqref{eq:splaygammaisgamma} and have coinciding start and end vertices. Thus they consist only of non-splitting directed paths from their start vertices to leaves. We show that such a path exists in $(\Gamma\Xi^+\Delta)^+$ if and only if a path with the same label exists in $(\Gamma\Delta)^+(\Gamma\Xi^+)^+$ from the start vertex to a leaf. Since both trees are determined by such paths, this is sufficient to show equality.

    Suppose a path labelled $w$ exists in $(\Gamma\Xi^+\Delta)^+ = \overline{\splay\left((\Gamma\Xi^+\Delta)^{(+)}\right)}$ from the start vertex to a leaf. For ease of notation, we will call a path \textit{suitable} in a given tree if it is from the start vertex to a leaf and is labelled $w$. Now a suitable path exists in $\overline{\splay\left((\Gamma\Xi^+\Delta)^{(+)}\right)}$ if and only if a suitable path exists in $\splay\left((\Gamma\Xi^+\Delta)^{(+)}\right)$. By definition of $\splay$, this path exists if and only if a suitable path exists in $(\Gamma\Xi^+\Delta)^{(+)}$. Call this path $p$. By Lemma \ref{lem:altproduct} and using that retraction is a $(2,1,0)$-morphism from $\lut(X) \to \fleh(X)$ (see \cite[Theorem 4.5]{kambites:free}), this tree is exactly $\overline{\Gamma \times \Xi^+ \times \Delta}^{(+)}$. Note that $p$ exists if and only if a suitable path exists in $\overline{\Gamma \times \Xi^+ \times \Delta}$, and in turn in $\Gamma \times \Xi^+ \times \Delta$. There are three cases for such a path: either it exists wholly in $\Gamma$, follows the trunk of $\Gamma$ and then ends at a leaf in $\Xi^+$, or follows the trunk of $\Gamma$ and ends at a leaf in $\Delta$. In all three cases, a suitable path exists in $(\Gamma \times \Delta)^{(+)} \times (\Gamma \times \Xi^+)^{(+)}$, and in fact all suitable paths are a result of one of these three cases. By the reverse argument to the above, this is equivalent to a suitable path existing in \begin{equation}\label{eq:bigtree}
        \overline{\splay((\Gamma \times \Delta)^{(+)})} \times \overline{\splay((\Gamma \times \Xi^+)^{(+)})}.
    \end{equation} Since there is a retraction on $(\Gamma \times \Delta)^{(+)}$ with image $\overline{\Gamma \times \Delta}^{(+)} = (\Gamma\Delta)^{(+)}$, and similar for $(\Gamma \times \Xi^+)^{(+)}$, the tree in \eqref{eq:bigtree} is exactly \[\overline{\splay((\Gamma\Delta)^{(+)})} \times \overline{\splay((\Gamma\Xi^+)^{(+)})} = (\Gamma\Delta)^+\times(\Gamma\Xi^+)^+\]by Lemma \ref{lem:splay}\eqref{lem:splay:splayrhogamma}. A suitable path exists here if and only if it exists in $\overline{(\Gamma\Delta)^+\times(\Gamma\Xi^+)^+}$, which is exactly $(\Gamma\Delta)^+(\Gamma\Xi^+)^+$ by Lemma \ref{lem:altproduct} as required.
 \end{proof}

Combining the previous two results, we have the following.

\begin{cor}\label{cor:hislefth}
    The $(2,1,0)$-algebra $\h(X)$ is a left $h$-adequate monoid generated by $X$.
\end{cor}

\subsection{Freeness} We now show that the left $h$-adequate monoid $\h(X)$ is the free left $h$-adequate monoid on $X$. In particular, we aim to show that given any left $h$-adequate monoid $M$ and any function $\chi: X \to M$, there exists a unique morphism $\h(X) \to M$ extending $\chi$.

Our desired morphism will be a restriction of the morphism $\rho$ previously constructed by Kambites \cite{kambites:freeleft}. Its exact definition is technical, relying on another morphism $\tau$ defined only on idempotents. We recall the definitions of both maps here, whilst encouraging the reader to consult \cite[Section 3]{kambites:freeleft} for comprehensive discussions. Note that we diverge from the original notation used in \cite{kambites:freeleft} to our forthcoming \textit{cone} notation.

Let $B$ be a left $X$-tree. Suppose $e$ is an edge of $B$. We define the \textit{cone of $B$ at $e$}, denoted $\cone_B(e)$, to be the left $X$-tree obtained by relocating both the defined start and end vertices of $B$ to the terminal vertex of $e$, and then taking the largest subgraph reachable from this distinguished vertex.

Given a fixed left Ehresmann monoid $M$ and fixed map $\chi: X \to M$, we define $\tau:\e(X) \to M$ recursively as follows. We define $\tau(\varepsilon) = 1_M$. For any other $B \in \e(X)$, we define\[\tau(B) := \prod_{\substack{e \text{ is an edge of }B\\ \text{with initial vertex}\\ \text{the start vertex of }B}} [\chi(\lambda(e))\,\tau(\cone_B(e))]^+\]where $\lambda(e)$ is the label of the edge $e$. 

We now use $\tau$ to define a map $\rho: \lut(X) \to M$. Given $\Gamma \in \lut(X)$ with branch decomposition $B_0 \times x_1 \times B_1 \times \dots \times x_k \times B_k$, we define\[\rho(\Gamma) := \tau(B_0)\chi(x_1)\tau(B_1)\cdots\chi(x_k)\tau(B_k).\]Clearly $\rho$ extends $\tau$ -- note that a tree $B \in \e(X)$ has branch decomposition simply $B$. These maps were studied in \cite[Section 3]{kambites:freeleft}. We direct the reader there for proofs of the following.

\begin{thm}[{\cite[Section 3]{kambites:freeleft}}]\label{thm:rhoprops}
    Fix a non-empty set $X$, a left Ehresmann monoid $M$, and a function $\chi: X \to M$. Define $\tau$ and $\rho$ as above.
    \begin{enumerate}
        \item The map $\rho$ extends $\chi$.\label{thm:rhoprops:extend}
        \item The map $\rho$ is a $(2,1,0)$-morphism.\label{thm:rhoprops:mor}
        \item For any $\Gamma,\Delta \in \e(X)$, we have $\tau(\Gamma \times \Delta) = \tau(\Gamma)\tau(\Delta)$.\label{thm:rhoprops:taumor}
        \item For any $\Gamma \in \lut(X)$, we have $\rho(\Gamma) = \rho\left(\overline{\Gamma}\right)$.\label{thm:rhoprops:bar}
    \end{enumerate}
\end{thm}

We now consider the map $\rho$ when $M$ further satisfies the left $h$ condition.

\begin{lem}\label{lem:rhosplay}
    Fix $X$, $M$ and $\chi$ as above. If $M$ satisfies the left $h$ condition, then $\rho(\Gamma) = \rho(\splay(\Gamma))$ for any $\Gamma \in \lut(X)$.
\end{lem}

\begin{proof}
    We first claim that $\tau(B) = \tau(\splay(B))$ for any $B \in \e(X)$.
    
    Let $B \in \e(X)$ be given. Define the \textit{height} of $B$ to be a maximal number of edges in a directed path with initial vertex the start vertex of $B$. We show that $\tau(B) = \tau(\splay(B))$ by induction on the height of $B$. Certainly if $B$ is of height at most $1$, then $B = \splay(B)$ and the result holds immediately. Now suppose that $\tau(C) = \tau(\splay(C))$ for all $C \in \e(X)$ of height at most $k$, for some $k \geq 1$.

    Let $B \in \e(X)$ have height $k + 1$. Let $e_1,\dots,e_l$ be the edges of $B$ with initial vertex the start vertex of $B$. By definition, $\tau(B) = \prod_{i=1}^l(\chi(\lambda(e_i))\,\tau(\cone_B(e_i))\,)^+$. For any vertex $v$ of $B$, excluding the start vertex, there is a unique edge $e_v \in \{e_1,\dots,e_l\}$ for which $v$ is reachable from the terminal vertex of $e_v$. Let $p_v$ be the left $X$-tree consisting of the directed path from the terminal vertex of $e_v$ to $v$ (with start and end vertex defined as the initial vertex of this path). Note then that \[\tau(\splay(B)) = \tau\Bigl( \bigtimes_{\substack{v\text{ is a}\\\text{leaf of }B}} (e_v\times p_v)^{(+)} \Bigr).\]By definition, observe that $\tau((e_v \times p_v)^{(+)}) = (\chi(\lambda(e_v))\,\tau(p_v))^+$. Thus by Theorem \ref{thm:rhoprops}\eqref{thm:rhoprops:taumor}, it follows that \[\tau(\splay(B)) = \prod_{\substack{v\text{ is a}\\\text{leaf of }B}} \tau((e_v\times p_v)^{(+)}) = \prod_{\substack{v\text{ is a}\\\text{leaf of }B}}(\chi(\lambda(e_v))\,\tau(p_v))^+.\]Since projections commute, we may partition this product according to the edges $e_v$ to see that\[\tau(\splay(B)) = \prod_{i=1}^l \prod_{\substack{v\text{ is a}\\\ \substack{\text{leaf of }B}\\\text{with }e_v = e_i}} (\chi(\lambda(e_i))\,\tau(p_v))^+.\]Consider this inner product. By applying the left $h$ condition recursively, we have that \[\prod_{\substack{v\text{ is a}\\\ \substack{\text{leaf of }B}\\\text{with }e_v = e_i}} (\chi(\lambda(e_i))\,\tau(p_v))^+ = \Bigl(\chi(\lambda(e_i)) \cdot \prod_{\substack{v\text{ is a}\\\ \substack{\text{leaf of }B}\\\text{with }e_v = e_i}} \tau(p_v) \Bigr)^+.\]This is, by Theorem \ref{thm:rhoprops}\eqref{thm:rhoprops:taumor}, exactly,\[\Bigl(\chi(\lambda(e_i)) \cdot \prod_{\substack{v\text{ is a}\\\ \substack{\text{leaf of }B}\\\text{with }e_v = e_i}} \tau(p_v) \Bigr)^+ = \left(\chi(\lambda(e_i))\,\tau(\splay(\cone_B(e_i)))\,\right)^+.\] Since the height of $\splay(\cone_B(e_i))$ is exactly the height of $\cone_B(e_i)$, which is at most $k$, we may appeal to our inductive hypothesis to see that\[\tau(\splay(\cone_B(e_i))) = \tau(\cone_B(e_i)).\]Hence overall\[\tau(\splay(B)) = \prod_{i=1}^l (\chi(\lambda(e_i))\,\tau(\cone_B(e_i)))^+ = \tau(B)\]as claimed, concluding our induction. We now show the result for $\rho$. 
    
    Let $\Gamma \in \lut(X)$ have branch decomposition $B_0 \times x_1 \times B_1 \times \dots \times x_k \times B_k$. Then\begin{align*}
        \rho(\splay(\Gamma)) &= \rho(\splay(B_0) \times x_1 \times \splay(B_1) \times \dots \times x_k \times \splay(B_k))\\
        &= \tau(\splay(B_0))\chi(x_1)\tau(\splay(B_1))\cdots\chi(x_k)\tau(\splay(B_k))\\
        &= \tau(B_0)\chi(x_1)\tau(B_1)\cdots\chi(x_k)\tau(B_k)\\
        &= \rho(\Gamma).\qedhere
    \end{align*}
\end{proof}

It follows from Theorem \ref{thm:rhoprops}\eqref{thm:rhoprops:bar} and Lemma \ref{lem:rhosplay} that $\rho\left(\widehat{\Gamma}\right) = \rho(\Gamma)$. We are now in a position to show the following crucial result, recalling that one may consider $\h(X)$ as a subset of $\lut(X)$.

\begin{thm}\label{thm:rhohat}
    Fix a non-empty set $X$, a left $h$-adequate monoid $M$ and a map $\chi:X \to M$. Define $\rho$ as above, and let $\widehat{\rho}$ be the restriction of $\rho : \lut(X) \to M$ to $\h(X)$. Then $\hat{\rho} : \h(X) \to M$ is a $(2,1,0)$-morphism.
\end{thm}

\begin{proof}
    Let $\Gamma,\Delta \in \h(X)$. We leverage Theorem \ref{thm:rhoprops} and Lemma \ref{lem:rhosplay} to show our required properties. Firstly,\begin{align*}
        \hat\rho(\Gamma\Delta) &= \rho(\Gamma\Delta) &&\\
        &= \rho\left(\overline{\Gamma \times \Delta} \right) && \text{(by Lemma \ref{lem:altproduct})}\\
        &= \rho(\Gamma \times \Delta) && \text{(by Theorem \ref{thm:rhoprops}\eqref{thm:rhoprops:bar})}\\
        &= \rho(\Gamma)\rho(\Delta) && \text{(by Theorem \ref{thm:rhoprops}\eqref{thm:rhoprops:mor})}\\
        &= \hat\rho(\Gamma)\hat\rho(\Delta).
    \end{align*}Secondly,\begin{align*}
        \hat\rho(\Gamma^+) &= \rho(\Gamma^+) &&\\
        &= \rho\left(\overline{\splay(\Gamma^{(+)})}\right) && \\
        &= \rho\left(\splay\left(\Gamma^{(+)}\right)\right) && \text{(by Theorem \ref{thm:rhoprops}\eqref{thm:rhoprops:bar})}\\
        &= \rho(\Gamma^{(+)}) && \text{(by Lemma \ref{lem:rhosplay})}\\
        &= \rho(\Gamma)^+ && \text{(by Theorem \ref{thm:rhoprops}\eqref{thm:rhoprops:mor})}\\
        &= \widehat{\rho}(\Gamma)^+. 
    \end{align*}Finally, we certainly have $\widehat{\rho}(\varepsilon) = \rho(\epsilon) = 1_M$.
\end{proof}

\begin{cor}
    For any non-empty set $X$, the $(2,1,0)$-algebra $\h(X)$ is the free left $h$-adequate monoid on $X$.
\end{cor}

\begin{proof}
    By Corollary \ref{cor:hislefth}, $\h(X)$ is a left $h$-adequate monoid. Moreover, for any left $h$-adequate monoid $M$ and map $\chi : X \to M$, define $\widehat{\rho}$ to be the $(2,1,0)$-morphism as in Theorem \ref{thm:rhohat}. Observe that for any $x \in X$, certainly $\widehat{\rho}(x) = \rho(x) = \chi(x)$ by Theorem \ref{thm:rhoprops}\eqref{thm:rhoprops:extend}. Since $\h(X)$ is $X$-generated (Corollary \ref{cor:hislefth}), it follows that $\widehat{\rho}$ is entirely determined by its restriction to $X$ and so must be the unique morphism with the required properties.
\end{proof}

\section{Properties}\label{sec:props}

We collect here some properties of free left $h$-adequate monoids using the description given in Section \ref{sec:splaying}.

\subsection{Other free objects}

We first note that the construction of the morphism $\widehat{\rho}$ in Theorem \ref{thm:rhohat} remains valid when $M$ is only left $h$-Ehresmann. As such, we immediately observe the following.

\begin{cor}\label{cor:hEhrandhAdcoincide}
    The free left $h$-Ehresmann monoid and the free left $h$-adequate monoid coincide in all ranks.
\end{cor}

By Proposition \ref{prop:adjoiningidentity}, we also obtain a description for free left $h$-Ehresmann/adequate semigroups and monoids.

\begin{cor}
    For any non-empty set $X$, the free left $h$-Ehresmann semigroup on $X$ is exactly the $(2,1)$-algebra $\h(X) \setminus \{\varepsilon\}$.
\end{cor}

\begin{cor}
    The free left $h$-Ehresmann semigroup and the free left $h$-adequate semigroup coincide in all ranks.
\end{cor}

Finally, one can consider the dual of the whole of Section \ref{sec:splaying} to obtain a similar description of free right $h$-Ehresmann and $h$-adequate semigroups. The unary operations required are ${}^{(\ast)}$ on \textit{right $X$-trees} (as seen in \cite{kambites:freeleft}) and then $\Gamma^\ast := \overline{\splay(\Gamma^{(\ast)})}$ where leaves are interpreted instead as sources and the operation $\splay$ is interpreted for right $X$-trees in the natural way.

\subsection{Adequacy}

We have seen that free left $h$-Ehresmann monoids are left $h$-adequate. Since free left Ehresmann monoids are two-sided adequate, one may expect free left $h$-Ehresmann monoids to be two-sided $h$-adequate. Indeed, this was observed by Fountain \cite[Corollary 3.4]{fountain:h} -- we provide an alternative proof here using our tree framework. Our approach has the advantages of exactly describing the unary operation ${}^\ast$, along with an easier visualisation of the element $\Gamma^\ast$ from a given tree $\Gamma$.

\begin{thm}\label{thm:hishadequate}
    Free left $h$-adequate monoids are $h$-adequate.
\end{thm}

\begin{proof}
    Fix a set $X$. Let $\Delta \in \lut(X)$ and suppose $\Delta$ has branch decomposition \[C_0 \times y_1 \times C_1 \times \dots \times y_l \times C_l.\]Define the left $X$-tree $\Delta^\ast := C_l$. When restricted to $\fleh(X)$, this is exactly the unary operation described in Theorem \ref{thm:flehisrhad} which imbues $\fleh(X)$ with an Ehresmann and a right $h$-adequate structure.
    
    Suppose further that $\Delta \in \h(X)$. By considering the branch decomposition of $\Delta$, we see that \[\overline{\splay(\Delta^\ast)} = \splay(\Delta^\ast) = \Delta^\ast\]and thus $\Delta^\ast \in \h(X)$. Hence ${}^\ast$ is a unary operation on $\h(X)$. We will show that the $(2,1,0)$-algebra $\h(X)$ is $h$-adequate when further equipped with ${}^\ast$.

    Recall from Theorem \ref{thm:hatmor} that the map $\widehat{\cdot} : \fleh(X) \to \h(X)$ is a surjective $(2,1,0$)-morphism. We claim that this morphism also respects the unary operation ${}^\ast$ defined on both structures. Indeed, let $\Gamma \in \fleh(X)$ and suppose $\Gamma$ has branch decomposition $B_0 \times x_1 \times B_1 \times \dots \times x_k \times B_k$. By considering the branch decomposition of $\widehat{\Gamma}$, one sees that $\left(\widehat{\Gamma}\right)^\ast$ is exactly $\overline{\splay(B_k)} = \widehat{B_k} = \widehat{\Gamma^\ast}$. Our claim is proved.

    Since right $h$-Ehresmann monoids form a variety, it immediately follows from Theorem \ref{thm:flehisrhad} and Birkhoff's Theorem \cite[Theorem 10]{birkhoff:struct} that $\h(X)$ is right $h$-Ehresmann. By Corollary \ref{cor:hislefth}, $\h(X)$ is left $h$-adequate and it only remains to show right adequacy. We show the two required quasi-identities.

    Let $\Gamma \in \h(X)$ and suppose $\Gamma^2 = \Gamma$. By left adequacy, we have that $\Gamma = \Gamma^+ = \overline{\splay(\Gamma^{(+)})}$ and hence $\Gamma$ has coinciding start and end vertex. Thus the branch decomposition of $\Gamma$ is $\Gamma$ itself, and so $\Gamma^\ast = \Gamma$.

    Let $\Gamma,\Delta,\Xi \in \h(X)$ and suppose that $\Xi\Gamma = \Xi\Delta$, i.e. $\overline{\Xi \times \Gamma} = \overline{\Xi \times \Delta}$ by Lemma \ref{lem:altproduct}. We aim to show that $\Xi^\ast\Gamma = \Xi^\ast\Delta$, i.e. $\overline{\Xi^\ast \times \Gamma} = \overline{\Xi^\ast \times \Delta}$. Note that if $\Xi$ has no trunk edges, then $\Xi = \Xi^\ast$ and we are immediately done. So suppose $\Xi$ has at least one trunk edge. Let $D_0 \times z_1 \times D_1 \times \dots \times z_m \times \Xi^\ast$ be the branch decomposition of $\Xi$, and for ease of notation write $\Xi_0 := D_0 \times z_1 \times D_1 \times \dots \times z_{m-1} \times D_{m-1} \times z_m$. By our assumption, we have that\[\overline{\Xi_0 \times \Xi^\ast \times \Gamma} = \overline{\Xi_0 \times \Xi^\ast \times \Delta}.\]Since in $\Xi_0$ there are no edges with initial vertex the end vertex of $\Xi_0$, any retraction on $\Xi_0 \times \Xi^\ast \times \Gamma$ restricts to a retraction on $\Xi^\ast \times \Gamma$, as any edges/vertices in this subtree must map to edges/vertices in the subtree and the restriction is certainly an idempotent morphism fixing the rooted vertices. By considering a retraction of $\Xi_0 \times \Xi^\ast \times \Gamma$ with pruned image, one sees this restricts to a retraction of $\Xi^\ast \times \Gamma$ with image $\overline{\Xi^\ast \times \Gamma}$. Hence if $\overline{\Xi_0 \times \Xi^\ast \times \Gamma} = \overline{\Xi_0 \times \Xi^\ast \times \Delta}$, then by performing this restriction to both sides we see that $\overline{\Xi^\ast \times \Gamma} = \overline{\Xi^\ast \times \Delta}$ indeed.

    Thus $\h(X)$ satisfies the right adequate quasi-identities and is therefore $h$-adequate.
\end{proof}

\subsection{Finitary conditions}\label{sec:finitary}

Some finitary conditions of free left $h$-adequate monoids were observed by Fountain \cite{fountain:h}, in particular being \textit{residually finite} (in the class of left adequate monoids) and \textit{hopfian}.

Recent interest in free left Ehresmann monoids has included various other finitary conditions, including being (\textit{weakly}) \textit{coherent}, \textit{finitely equated} and \textit{ideal Howson} \cite{gould:coherent}. We conclude this section with an investigation of these properties for $\h(X)$. Whilst deep understanding of the forthcoming concepts, presentations, ideals and actions is not required, we direct unfamiliar readers to \cite{dasar:coherent,gould:coherent,howie:fundamentals} amongst many other texts for undefined terminology.

A monoid $M$ is \textit{right coherent} if every finitely generated subact of every finitely presented right $M$-act is itself finitely presented. Since right ideals of $M$ are precisely the subacts of the right action of $M$ on itself, we may weaken the condition and say $M$ is \textit{weakly right coherent} if every finitely generated right ideal of $M$ is finitely presented (as a right $M$-act).

For each element $a \in M$, the \textit{right annihilator congruence of $a$} is the congruence $\mathbf{r}(a)$ on $M$ defined by $u\mathrel{\mathbf{r}(a)}v$ if and only if $au = av$. We say $M$ is \textit{finitely right equated} if $\mathbf{r}(a)$ is finitely generated (as a congruence) for all $a \in M$.

A monoid is said to be \textit{right ideal Howson} if the intersection of any two principal right ideals is finitely generated. This condition is equivalent to having finite generation of the intersection of any two finitely generated right ideals.

In our investigation of these properties for $\h(X)$, we will make good use of the following result.

\begin{prop}[{\cite[Lemma 3.5]{gould:coherentmons}, \cite[Proposition 2.19]{gould:coherent}}]\label{prop:coherentconditions}
    Let $M$ be a right adequate monoid. Then:
    \begin{enumerate}
        \item $M$ is finitely right equated.
        \item $M$ is right ideal Howson if and only if $M$ is weakly right coherent.
    \end{enumerate}
\end{prop}

In fact Proposition \ref{prop:coherentconditions} only requires $M$ to be \textit{right abundant} (see \cite[Definition 2.11]{gould:coherent}), which right adequate monoids indeed are. There are, of course, duals of all of the above (including Proposition \ref{prop:coherentconditions}) replacing ``right'' with ``left''. 

Gould and Johnson \cite{gould:coherent} apply Proposition \ref{prop:coherentconditions} to free left Ehresmann monoids. For any set $X$, it follows immediately that $\fleh(X)$ is finitely right equated and it is shown in \cite[Propositions 6.4 and 6.7]{gould:coherent} that $\fleh(X)$ is left and right ideal Howson. We observe the following analogue.

\begin{prop}\label{prop:hisRIH}
    Free left $h$-Ehresmann monoids are left and right ideal Howson.
\end{prop}

\begin{proof}
    The free left $h$-Ehresmann monoid of rank $0$ is certainly left and right ideal Howson. So now fix a non-empty set $X$. By considering $\h(X)$ as a subset of $\fleh(X)$, it follows from Lemma \ref{lem:altproduct} that $\h(X)$ is a submonoid of $\fleh(X)$. By Theorem \ref{thm:hatmor}, the map $\widehat{\cdot} : \fleh(X) \to \h(X)$ is a surjective $(2,1,0)$-morphism. In particular it is a surjective morphism of semigroups, and it satisfies $\widehat{\Gamma} = \Gamma$ for all $\Gamma \in \h(X)$ by \eqref{eq:hathatishat}. Such morphisms are known as \textit{retracts} (not to be confused with branch retraction) and the class of [resp. left] right ideal Howson monoids is known to be closed under taking retracts \cite[Corollary 3.5]{dasar:coherent}. Since $\fleh(X)$ is left ideal Howson \cite[Proposition 6.4]{gould:coherent} and right ideal Howson \cite[Proposition 6.7]{gould:coherent}, our result follows.
\end{proof}

We obtain a few extra properties for free.

\begin{cor}
    Free left $h$-Ehresmann monoids are finitely left equated, finitely right equated, weakly left coherent and weakly right coherent.
\end{cor}

\begin{proof}
    Follows from Theorem \ref{thm:hishadequate}, Proposition \ref{prop:coherentconditions} and its dual, and Proposition \ref{prop:hisRIH}.
\end{proof}

We finish this section by commenting on full coherency. By \cite[Theorem 6.3]{gould:coherentfreemons}, the class of [resp. left] right coherent monoids is closed under taking retracts. Unfortunately we are unable to appeal to our constructed retract in this case: free left Ehresmann monoids are known to not be left coherent for rank at least $2$ \cite[Theorem 5.5]{gould:coherent} and whether they are right coherent is as yet unknown. We too observe the following analogue.

\begin{thm}
    For rank at least $2$, free left $h$-Ehresmann monoids are not left coherent.
\end{thm}

\begin{proof}
    In the proof of Proposition \ref{prop:hisRIH}, we saw that one may treat $\h(X)$ as a submonoid of $\fleh(X)$. In turn, $\fleh(X)$ may be considered as a submonoid of the free two-sided Ehresmann monoid on $X$. By taking $a \neq b \in X$, an entirely identical argument to that of \cite[Theorem 5.5]{gould:coherent} shows that $\h(X)$ is not left coherent.
\end{proof}

\section{Open Questions}\label{sec:openqs}

As previously mentioned, it is open whether free left Ehresmann monoids are right coherent. On the other hand, free left ample monoids in all ranks are right coherent \cite[Theorem 5.7]{gould:coherentfreemons}. With left $h$-Ehresmann monoids sitting between these classes, a natural question arises.

\begin{quest}
    Are free left $h$-Ehresmann monoids right coherent?
\end{quest}

In particular if $\fleh(X)$ is found to be right coherent, then by using our retract as in the proof of Proposition \ref{prop:hisRIH} one may conclude that $\h(X)$ is right coherent. In the monogenic case, since the free left $h$-Ehresmann monoid is exactly the free left Ehresmann monoid, both left and right coherency remain open \cite[Question 8.4]{gould:coherent}.

Since the work here (and its dual) describe free left and free right $h$-Ehresmann semigroups, one might expect a description of free two-sided $h$-Ehresmann semigroups to nicely follow. However the obvious idea of splaying and retracting Kambites' $X$-trees from \cite{kambites:free} unfortunately fails -- as in \cite{aird:growth}, complications arise with so-called ``zig-zag trees''. Consider the following example. One may show that the words $(ab^+c)^+$ and $(ab)^+(ac(b^+c)^\ast)^+$ are identified in the free $h$-Ehresmann monoid on $\{a,b,c\}$, but representing these words as trees and performing the sensible notion of splaying and retraction fails to identify them.

\begin{quest}
    Is there a description of free $h$-Ehresmann monoids using directed trees?
\end{quest}

We close on an open-ended question which we hope the reader entertains.

\begin{quest}
    Which properties of (free, left) ample semigroups translate to (free, left) $h$-adequate or $h$-Ehresmann semigroups?
\end{quest}


\section*{Acknowledgements}
The author would like to thank Thomas Aird for helpful discussions.



\end{document}